\documentclass[12pt]{article}
\usepackage[margin=1cm]{caption}
\usepackage{mathrsfs}
\usepackage{mathptmx}       
\usepackage{helvet}         
\usepackage{courier}        
\usepackage{type1cm}        
\usepackage{makeidx}         
\usepackage{graphicx}        
\usepackage{multicol}        
\usepackage[bottom]{footmisc}
\usepackage{amsmath}
\usepackage[normalem]{ulem}
\usepackage{tikz-cd}
\usepackage{amsfonts}
\usepackage{amssymb}
\usepackage{amsthm}
\usepackage{amsopn}
\usepackage{subcaption}
\usepackage{comment}
\newtheorem{thm}{Theorem}
\newtheorem{lem}[thm]{Lemma}
\newtheorem{rk}[thm]{Remark}
\newtheorem{cor}[thm]{Corollary}
\newtheorem{prop}[thm]{Proposition}
\usepackage[top=2cm, bottom=2cm, left=2cm, right=2cm]{geometry}
\numberwithin{thm}{section}
\numberwithin{figure}{section}
\numberwithin{equation}{section}
\begin{document}
\newcommand*{\TitleFont}{%
\fontsize{19}{20}%
\selectfont}
\title{\TitleFont{Field-measure correspondence in Liouville quantum gravity almost
surely commutes with all conformal maps simultaneously}}

\author{Scott Sheffield\thanks{\noindent Department of Mathematics, Massachusetts Institute of Technology, Cambridge, MA, sheffield@math.mit.edu}\quad
Menglu Wang\thanks{\noindent Department of Mathematics, Massachusetts Institute of Technology, Cambridge, MA, mengluw@math.mit.edu}} 
\date{}
\maketitle

\abstract{In Liouville quantum gravity (or $2d$-Gaussian
multiplicative chaos) one seeks to define a measure $\mu^h = e^{\gamma
h(z)} dz$ where $h$ is an instance of the Gaussian free field on a
planar domain $D$. Since $h$ is a distribution, not a function, one
needs a regularization procedure to make this precise: for example,
one may let $h_\epsilon(z)$ be the average value of $h$ on the circle of radius $\epsilon$ centered at $z$ (or an analogous average defined using a bump function supported inside that circle) and then
write $\mu^h = \lim_{\epsilon \to 0} \epsilon^{\frac{\gamma^2}{2}} e^{\gamma
h_\epsilon(z)} dz$.

If $\phi: \tilde D \to D$ is a conformal map, one can write $\tilde h
= h \circ \phi + Q \log |\phi'|$, where $Q = 2/\gamma + \gamma/2$. The
measure $\mu^{\tilde h}$ on $\tilde D$ is then a.s.\ equivalent to the
pullback via $\phi^{-1}$ of the measure $\mu^h$ on $D$. Interestingly,
although this a.s.\ holds for each \textit{given} $\phi$, nobody has
ever proved that it a.s.\ holds \textit
{simultaneously} for all possible $\phi$. We will prove that this is indeed the case. This is
conceptually important because one frequently defines a \textit{quantum
surface} to be an equivalence class of pairs $(D, h)$ (where pairs
such as the $(D,h)$ and $(\tilde D, \tilde h)$ above are considered
equivalent) and it is useful to know that the set of pairs $(D,\mu^{h})$ obtained from the set of pairs $(D,h)$ in an equivalence class is itself an equivalence class with respect to the usual measure pullback relation.}

\section{Introduction}\label{section::definition}
In dimension $d$, a Gaussian multiplicative chaos is a random measure on a domain $D\subset\mathbb{R}^{d}$ that can be formally written as
\begin{equation}\label{equation::chaos}
\mu(dx)=e^{\gamma h(x)-\frac{\gamma^{2}}{2}\mathbb{E}h^{2}(x)}\sigma(dx),
\end{equation}
where $\gamma>0$, $h$ is a centered Gaussian field and $\sigma$ is a Radon measure on $D$. We assume that $h$ possesses a covariance kernel of the form
\begin{equation*}
\mathbb{E}(h(x)h(y))=-\log(|x-y|\wedge 1)+g(x,y),
\end{equation*}
where $g$ is a continuous bounded function on $D\times D$. Measures of this form first appeared in H{\o}egh-Krohn's work \cite{MR0292433} in the setting where $h$ is the $2d$-massive free field and $\sigma$ is the Lebesgue measure. H{\o}egh-Krohn showed that the measure $\mu$ exists and is non-trivial for $\gamma\in[0,\sqrt{2})$. In \cite{MR829798} Kahane introduced the term \textit{Gaussian multiplicative chaos} and further developed the theory in order to obtain a continuous counterpart of the multiplicative cascades proposed by Mandelbrot in \cite{FLM:385758}. In particular, Kahane (apparently unaware of the work in \cite{MR0292433}) extended the construction in \cite{MR0292433} from $\gamma\in[0,\sqrt{2})$ to $\gamma\in[0,2)$ and expanded the theory in many ways. This kind of random measure has applications in a variety of fields like Schramm-Loewner Evolution \cite{duplantier2011schramm, sheffield2010conformal}, Liouville quantum gravity \cite{DuplantierSheffieldLQGKPZ, miller2013quantum} and $3d$-turbulence \cite{mandelbrot1972possible, fyodorov2010freezing}. For a thorough discussion on Gaussian multiplicative chaos and the relatively recent developments, we refer the reader to the survey \cite{MR3274356} and the references therein.

In this article, we will focus on the Liouville quantum gravity measure, which is a special case of the Gaussian multiplicative chaos where $d=2$ and $h$ is the Gaussian free field (see Remark \ref{remark::case} for the choice of $\sigma(dx)$). Note that the Gaussian free field is not defined pointwise, but is almost surely a distribution. One needs a regularization procedure to make (\ref{equation::chaos}) precise.  

The first approach is to apply Kahane's theory (\cite{MR829798}). Roughly speaking, since the covariance kernel of the Gaussian free field is $\sigma$-positive, there is a sequence of continuous Gaussian processes $(X_{n})_{n}$ with independent increments $X_{n+1}-X_{n}$ such that the covariance kernels of $X_{n}$ converge to the covariance kernel of $h$. Basic martingale theory implies that it is almost surely the case that the random measures 
\[\mu_{n}(dz):=e^{\gamma X_{n}(z)-\frac{\gamma^{2}}{2}\mathbb{E}X^{2}_{n}(z)}\sigma(dz)\]
converge weakly to a random measure $\mu$. In \cite{MR829798}, a uniqueness result was also proved: the law of the limiting measure $\mu$ does not depend on the sequence $(X_{n})_{n}$. Moreover, Kahane showed that the limiting measure $\mu$ is non-degenerate if and only if $\gamma<2$. Note that Kahane's theory only ensures equality in \textit{law}, i.e., it does not show that there is almost surely a unique way to produce a measure $\mu^{h}$ from a given instance $h$ of the GFF.

The second approach is to apply convolution techniques. Given an instance $h$ of the Gaussian free field, we consider the convolution of $h$ with a mollifier $f$. Suppose that $f:\mathbb{R}^{2}\rightarrow\mathbb{R}_{\ge 0}$ is a radially symmetric bump function compactly supported on $B_{1}(0)$ with \[\int_{B_{1}(0)}f(z)dz=1,\]
where $dz$ is the Lebesgue measure on $\mathbb{R}^{2}$.

For $0<\epsilon<1$, let
\[f_{\epsilon}(z):=\frac{1}{\epsilon^{2}}f(\frac{z}{\epsilon}),\]
and
\[h*f_{\epsilon}(z):=(h, f_{\epsilon}(z-\cdot)).\]
Consider the family of random measures 
\begin{equation}\label{equation::approximate}
\tilde{\mu}_{\epsilon}:=e^{\gamma h*f_{\epsilon}(z)-\frac{\gamma^{2}}{2}\mathbb{E}|h*f_{\epsilon}(z)|^{2}}\sigma(dz).
\end{equation}
It was established in \cite{MR2642887} that $\tilde{\mu}_{\epsilon}$ converge in \textit{law} in the space of Radon measures (equipped with the topology of weak convergence) towards a random measure as $\epsilon\rightarrow 0$. Note that in this case the approximating Gaussian fields $h*f_{\epsilon}(z)$ and the corresponding measures $\tilde{\mu}_{\epsilon}$ are all a.s.\ determined by $h$. The convergence in probability and convergence in $L^{p}$ of the random measures as $\epsilon\rightarrow 0$ were studied in \cite{MR3475456, junnila2015uniqueness, berestycki2015elementary}. It is known that (see e.g., \cite[Theorem 26]{MR3475456}) the random measures $\tilde{\mu}_{\epsilon}$ converge in \textit{probability} to a limiting measure $\mu$ and the random measure $\mu$ does not depend on the choice of the mollifier $f$. 

In \cite{DuplantierSheffieldLQGKPZ}, the authors set $f$ to be the uniform measure on the unit circle instead of a smooth function and proved the \textit{almost} \textit{sure} convergence of (\ref{equation::approximate}). To be concrete, let $D$ be a bounded simply connected domain in $\mathbb{R}^{2}$, and $h$ an instance of the zero boundary Gaussian free field (GFF) on $D$. Denote by $h_{\epsilon}(z)$ the average value of $h$ on the circle of radius $\epsilon$ centered at $z$ (see Section \ref{section::pre} for a quick overview). 

Fix $\gamma\in[0,2)$, and write 
\begin{equation}\label{equation::def}
\bar{h}_{\epsilon}(z):=\gamma h_{\epsilon}(z)+\frac{\gamma^{2}}{2}\log\epsilon.
\end{equation}
The Liouville quantum gravity measure on $D$ is the weak limit as $\epsilon\rightarrow 0$ of the measures $\mu_{\epsilon}:=e^{\bar{h}_{\epsilon}(z)}dz$. In \cite{DuplantierSheffieldLQGKPZ}, the authors showed that the limiting measure a.s.\ exists as $\epsilon\rightarrow 0$ along negative powers of $2$, which we denote by $\mu=\mu^{h}=e^{\gamma h(z)}dz$. There is in fact a significant technical difficulty in extending
from a.s.\ convergence along negative powers of $2$ to a.s.\ convergence when
$\epsilon$ is not restricted to negative powers of $2$. Overcoming that
difficulty requires a much deeper understanding of the continuity
properties of the map $z\mapsto h_\epsilon(z)$, and this is the first
part of what will be accomplished in the current paper:

\begin{thm}\label{theorem::measureconvergence}
Fix $\gamma\in[0,2)$ and define $D, h,\mu_{\epsilon},\mu$ as above. Then it is almost surely the case that as $\epsilon\rightarrow 0$, the measures $\mu_{\epsilon}$ converge weakly in $D$ to $\mu$.
\end{thm}

Next we consider the convolution of the GFF with a mollifier $f$ and the associated random measures $\tilde{\mu}_{\epsilon}$ as defined in (\ref{equation::approximate}). We will show that the same result as in Theorem \ref{theorem::measureconvergence} holds for $\tilde{\mu}_{\epsilon}$ as well:
\begin{thm}\label{prop::convergence}
Fix $\gamma\in[0,2)$ and define $D, h, \mu$, $\tilde{\mu}_{\epsilon}$ as above. Then it is almost surely the case that as $\epsilon\rightarrow 0$, the measures $\tilde{\mu}_{\epsilon}$ converge weakly in $D$ to $\mu$. The result remains true if we replace $h$ with a GFF with non-zero boundary conditions. 
\end{thm}
\begin{rk}
Note that $h$ a.s.\ determines the measure $\mu^{h}$. It was shown in \cite{berestycki2014equivalence} that the converse is true: $h$ is measurably determined by $\mu^{h}$. Therefore, we obtain a field-measure correspondence $\rho$ given by
\begin{equation}\label{equation::map}
h\stackrel{\rho}{\mapsto}\mu^{h}:=\lim_{\epsilon\rightarrow 0}e^{\gamma h*f_{\epsilon}(z)-\frac{\gamma^{2}}{2}\mathbb{E}|h*f_{\epsilon}(z)|^{2}}\sigma(dz).
\end{equation}
\end{rk}
Finally we study the transformation of the quantum measures under conformal maps. We say that $\phi:\tilde{D}\rightarrow D$ is a conformal map if $\phi$ is analytic, one to one, and onto.
Let 
\begin{equation}\label{equation::Lambda}
\Lambda:=\left\{\phi:\tilde{D}\rightarrow D; \textrm{ $\tilde{D}$ is a planar domain and $\phi$ is a conformal map from $\tilde{D}$ to $D$}\right\}
\end{equation} 
be the collection of the conformal maps onto $D$. If $h$ is an instance of the GFF on $D$,  then $h\circ\phi$ is a GFF on $\tilde{D}$ (see Proposition \ref{proposition::conformal}). In \cite{DuplantierSheffieldLQGKPZ}, the authors proved the following transformation rule. For each $\phi\in\Lambda$, let 
\begin{equation}\label{equation::tildeh}
\tilde{h}_{\phi}=h\circ\phi+Q\log|\phi'|, 
\end{equation}
where $Q=2/\gamma + \gamma/2$. Then it is almost surely the case that $\mu^{h}$ is the image under $\phi$ of $\mu^{\tilde{h}_{\phi}}$. That is, we have
\begin{equation}\label{equation::transformrule}
\mu^{h}(A)=\mu^{\tilde{h}_{\phi}}\left(\phi^{-1}(A)\right)
\end{equation}
for each Borel set $A\subset D$. In other words, for each \textit{given} conformal map $\phi\in\Lambda$, the field-measure correspondence $\rho$ as defined in (\ref{equation::map}) a.s.\ commutes with $\phi$ (see Figure \ref{figure::dia}). 
\begin{figure}[!hbp]
\begin{center}
\includegraphics[scale=0.7]{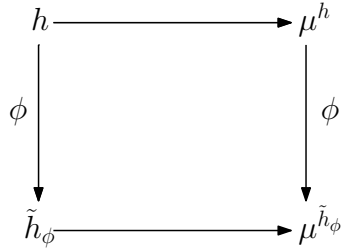}
\end{center}
\caption{\label{figure::dia}\normalsize{ In this commutative diagram, the two rightward arrows correspond to the field-measure correspondence $\rho$ in (\ref{equation::map}). The downward arrow on the left corresponds to the pullback of distributions given by (\ref{equation::tildeh}). The downward arrow on the right corresponds to the pullback of measures (\ref{equation::transformrule}).}}
\end{figure}

We will establish the fact that the transformation rule a.s. holds \textit{simultaneously} for all $\phi\in\Lambda$:
\begin{thm}\label{thm::conformalvariance}
Fix $\gamma\in[0,2)$ and define $D, h, \mu^{h}, \Lambda$ as above. Then it is almost surely the case that for all $\phi\in\Lambda$, the measures $\mu^{\tilde{h}_{\phi}}$ with $\tilde{h}_{\phi}$ as in (\ref{equation::tildeh}) are well-defined and the transformation rule (\ref{equation::transformrule}) holds simultaneously for all $\phi\in\Lambda$.
\end{thm}

Theorem \ref{thm::conformalvariance} implies that the quantum area measure can be established in a completely parameterization independent way. Like Theorem \ref{theorem::measureconvergence}, this result requires a deep and highly novel exploration of the properties of $h_\epsilon(z)$ and of the
corresponding measures. The proof combines general facts about
distributions with some results about the extrema of Gaussian process. This work is subtle and technical.
But we feel it is also quite important, as it puts the basic Liouville
quantum gravity equivalence relationship (involved in the very
definition of \textit{quantum surface} in \cite{DuplantierSheffieldLQGKPZ} and many other
papers) on a much more solid and satisfying foundation.

This article is organized as follows. In Section \ref{section::prelimi} we will provide some background on the GFF, jointly Gaussian random variables and conformal maps. The proofs of the theorems are contained in Section \ref{section::proof}. In Section \ref{section::boundary}, we will discuss the generalization to boundary measures. 
\medbreak
\noindent\textbf{Acknowledgment.} We would like to thank Ewain Gwynne, Nina Holden and Xin Sun for helpful comments on the preliminary version of this article.

\section{Preliminaries}\label{section::prelimi}
\subsection{GFF and the circle average processes}\label{section::pre}
In this section, we give a brief review of the construction of the GFF as well as some properties that will be important for us later. We refer the reader to $\cite{SheffieldGFFMath}$ for a detailed introduction.

Let $D$ be a bounded simply connected domain in $\mathbb{R}^{2}$, and let $H_{s}(D)$ denote the the space of $C^{\infty}$ real-valued functions compactly supported on $D$.

We let $H(D)$ be the Hilbert space closure of $H_{s}(D)$ equipped with the Dirichlet inner product:
\[(f_{1},f_{2})_{\nabla}:=\frac{1}{2\pi}\int_{D}\nabla f_{1}(z)\cdot\nabla f_{2}(z)dz.\]
The GFF on $D$ can be expressed as a random sum of the form $h=\sum_{n}\alpha_{n}f_{n}$, where the $f_{n}$ are an orthonormal basis of $H(D)$ and the $\alpha_{n}$ are i.i.d. standard Gaussian random variables. The sum a.s.\ does not converge in $H(D)$, but it a.s.\ converges in the space of distributions on $D$. In fact, for each $f\in H(D)$, we may define $(h,f)_{\nabla}$ as a zero-mean Gaussian random variable and for $f_{1}, f_{2}\in H(D)$, we have
\[\mathrm{cov}\left((h,f_{1})_{\nabla}, (h,f_{2})_{\nabla}\right)=(f_{1}, f_{2})_{\nabla}.\]

For $x, y\in D$, we let 
\begin{equation}\label{equation::green}
G(x,y):=-\log|y-x|-\tilde{G}_{x}(y),
\end{equation}
where $\tilde{G}_{x}(y)$ is the harmonic extension to $y\in D$ of the function of $y$ on $\partial D$ given by $-\log|y-x|$. Then $G(x,y)$ is the Green's function for the Laplacian on $D$, i.e., $\Delta G(x,\cdot)=-2\pi\delta_{x}(\cdot)$ for $x\in D$, with zero boundary conditions. It is non-negative on $D\times D$.

Integrating by parts implies that for $\rho_{1},\rho_{2}\in H_{s}(D)$, we have
\begin{eqnarray*}
\mathrm{cov}\left((h,\rho_{1}), (h,\rho_{2})\right)&=&(2\pi)^{2}\mathrm{cov}\left((h,-\Delta^{-1}\rho_{1})_{\nabla}, (h,-\Delta^{-1}\rho_{2})_{\nabla}\right)\\
&=&(2\pi)^{2}(-\Delta^{-1}\rho_{1},-\Delta^{-1}\rho_{2})_{\nabla}\\
&=&-2\pi\int_{D}\rho_{1}(x)\cdot\Delta^{-1}\rho_{2}(x)dx\\
&=&\iint_{D\times D}\rho_{1}(x)G(x,y)\rho_{2}(y)dxdy.\\
\end{eqnarray*}

One can also define the GFF with non-zero boundary conditions. Suppose that $D$ has smooth boundary. If $f:\partial D\rightarrow\mathbb{R}$ is a function that is $L^{1}$ with respect to the harmonic measure on $\partial D$ viewed from some point in $D$, and $F$ is its harmonic extension from $\partial D$ to $D$, then the law of a GFF on $D$ with boundary condition $f$ is given by $h+F$ where $h$ is a zero boundary GFF on $D$.

We record two important properties of the GFF in the following propositions.

\begin{prop}[Markov property]\label{proposition::markov}
Suppose that $h$ is an instance of the zero boundary GFF on $D$. For any open subset $W\subset D$, there is a decomposition of $h$:
\[h=h_{W}+h_{W}^{\perp},\]
where $h_{W}$ and $h_{W}^{\perp}$ are random distributions on $D$ such that

1. The restriction of $h_{W}$ to $W$ is a zero boundary GFF on $W$, and $h_{W}$ is zero outside of $W$.

2. $h_{W}^{\perp}$ is harmonic on $W$.

3. Define $\mathscr{F}_{W}:=\sigma\{(h_{W},f)_{\nabla}, f\in H(D)\}$ and $\mathscr{F}_{W}^{\perp}:=\sigma\{(h_{W}^{\perp},f)_{\nabla}, f\in H(D)\}$, then $\mathscr{F}_{W}$ and $\mathscr{F}_{W}^{\perp}$ are independent.
\end{prop}

\begin{rk}\label{remark::independence}
Suppose that $f\in H_{s}(D)$ is compactly supported in $W$ and $(h,f)=(h_{W},f)+(h_{W}^{\perp},f)$. Given $(h_{W}^{\perp},f)$, the Markov property implies that $(h,f)$ and $\mathscr{F}_{W}^{\perp}$ are conditionally independent. In particular, if $g\in H_{s}(D)$ is compactly supported in $W^{c}$, then $(h,g)=(h_{W}^{\perp},g)$ and it follows that $(h,f)$ and $(h,g)$ are conditionally independent given $(h_{W}^{\perp},f)$.
\end{rk}

\begin{prop}[Conformal invariance of the GFF]\label{proposition::conformal}
Suppose that $h$ is an instance of the zero boundary GFF on $D$ and $\phi:\tilde{D}\rightarrow D$ is a conformal map. We define $h\circ\phi$ to be a distribution on $\tilde{D}$ by $(h\circ\phi,\tilde{\rho}):=(h,|(\phi^{-1})'|^{2}\tilde{\rho}\circ\phi^{-1})$ for $\tilde{\rho}\in H_{s}(\tilde{D})$, where $(\phi^{-1})'$ is the complex derivative of $\phi^{-1}$. Then $h\circ\phi$ is a zero boundary GFF on $\tilde{D}$.
\end{prop}
  
Suppose that $h$ is an instance of the zero boundary GFF on $D$. Denote by $h_{\epsilon}(z)$ the average value of $h$ on the circle of radius $\epsilon$ centered at $z$. (For this definition, we assume that $h$ is identically zero outside of $D$.) For $z,y\in D$, let
\[\xi_{\epsilon}^{z}(y):=-\log (|z-y|\vee\epsilon)-\tilde{G}_{z,\epsilon}(y),\]
where $\tilde{G}_{z,\epsilon}(y)$ is the harmonic extension to $y\in D$ of the function of $y$ on $\partial D$ given by $-\log (|z-y|\vee\epsilon)$. Observe that as a distribution $\frac{1}{2\pi}(-\Delta\xi_{\epsilon}^{z})$ is equal to the uniform measure on $\partial B_{\epsilon}(z)$. Integrating by parts yields that
\[h_{\epsilon}(z)=(h,\xi_{\epsilon}^{z})_{\nabla}.\]
We summarize several properties of the circle average processes in the following proposition. The reader may consult \cite[Section 3.1]{DuplantierSheffieldLQGKPZ} for the proofs.

\begin{prop}\label{prop::circleaverage}

1. The process $h_{\epsilon}(z)$ has a modification which is a.s.\ locally $\eta$-H\"{o}lder continuous in the pair $(z,\epsilon)\in\mathbb{C}\times(0,\infty)$ for every $\eta<1/2$.

2. If $B_{\epsilon_{1}}(z_{1})$ and $B_{\epsilon_{2}}(z_{2})$ are disjoint and both contained in $D$, then \begin{equation}\label{eqaution::samecenter}
\mathrm{cov}\left(h_{\epsilon_{1}}(z_{1}), h_{\epsilon_{2}}(z_{2})\right)=G(z_{1},z_{2}).
\end{equation}

3. If $B_{\epsilon_{1}}(z)\subset D$ and $\epsilon_{1}\ge\epsilon_{2}$, then
\begin{equation}\label{equation::variance}
\mathrm{cov}\left(h_{\epsilon_{1}}(z), h_{\epsilon_{2}}(z)\right)=-\log\epsilon_{1}+\log C(z;D),
\end{equation}
where $C(z;D)$ is the conformal radius of $D$ viewed from $z$. That is, $C(z;D)=|\psi'(z)|^{-1}$, where $\psi: D\rightarrow\mathbb{D}$ is a conformal map to the unit disk $\mathbb{D}$ with $\psi(z)=0$.

4. Write $\mathcal{V}_{t}=h_{e^{-t}}(z)$, and $t_{0}^{z}=\inf\{t: B_{e^{-t}}(z)\subset D\}$. If $z\in D$ is fixed, then the law of $V_{t}:=\mathcal{V}_{t_{0}^{z}+t}-\mathcal{V}_{t_{0}^{z}}$ is that of a standard Brownian motion independent of $\mathcal{V}_{t_{0}^{z}}$.
\end{prop}
\begin{rk}\label{remark::case}
Combining (\ref{equation::def}) with (\ref{equation::variance}) implies that the Liouville quantum gravity measure $\mu$ is a special case of the Gaussian multiplicative chaos (\ref{equation::chaos}) where $h$ is the $2d$-GFF and $\sigma(dz)=C(z;D)^{\frac{\gamma^{2}}{2}}dz$.
\end{rk}

\subsection{Extrema of jointly Gaussian random variables}
In this section, we recall some results about the fluctuation and tail bounds of Gaussian extrema, which we will make use of in the proof of Theorem \ref{thm::conformalvariance}.
\begin{prop}\label{propsition::supest}
If $X_{1}, X_{2},\ldots$ is a sequence of centered jointly Gaussian random variables each with variance at most $\sigma^{2}$, such that $\sup X_{i}$ are a.s.\ finite, then
\begin{equation}\label{equation::supvar}
\mathrm{var}(\sup X_{i})\le\sigma^{2},
\end{equation}
and for any $r>0$, we have
\begin{equation}\label{equation::tailofsup}
\mathbb{P}(\left|\sup X_{i}-\mathbb{E}(\sup X_{i})\right|\ge r)\le 2e^{-\frac{r^{2}}{2\sigma^{2}}}.
\end{equation}
\end{prop}
(\ref{equation::supvar}) was proved by Houdr\'{e} \cite{MR1354013}. (\ref{equation::tailofsup}) is called Borel-TIB inequality, independently invented by Borel \cite{MR0399402} and Tsirelson, Ibragimov, and Sudakov \cite{cirel1976norms}.
\begin{rk}
By the symmetry of centered jointly Gaussian random variables, (\ref{equation::supvar}) and (\ref{equation::tailofsup}) also hold with $\inf X_{i}$ in place of $\sup X_{i}$.
\end{rk}
Recall that a centered random variable $X$ is said to be sub-Gaussian if there exists a constant $\sigma>0$ such that for any $r>0$, we have $\mathbb{P}[|X|\ge r]\le 2e^{-\frac{r^{2}}{2\sigma^{2}}}$. Therefore,
(\ref{equation::tailofsup}) implies that $\sup X_{i}-\mathbb{E}(\sup X_{i})$ is sub-Gaussian. It is well known (see e.g., \cite{MR0117506}) that sub-Gaussian random variables satisfy a Laplace transform condition:
\begin{prop}\label{proposition::subgaussian}
For a centered random variable $X$, if there exists a constant $\sigma>0$ such that for any $r>0$, we have
\[\mathbb{P}[|X|\ge r]\le 2e^{-\frac{r^{2}}{2\sigma^{2}}},\]
then for any $\alpha>0$, it holds that
\[\mathbb{E}\left(e^{\alpha X}\right)\le e^{4\alpha^{2}\sigma^{2}}.\]
\end{prop}
\begin{rk}
If $X\sim\mathscr{N}(\mu,\sigma^{2})$ is a Gaussian random variable, then we always have
\[\mathbb{E}\left(e^{\alpha X}\right)=e^{\alpha\mu+\frac{\alpha^{2}\sigma^{2}}{2}}.\]
Combining this with (\ref{equation::variance}) implies that when $\epsilon$ is small enough, we have $\mathbb{E}e^{\bar{h}_{\epsilon}(z)}=C(z,\;D)^{\frac{\gamma^{2}}{2}}$.
\end{rk}

\subsection{de Branges's Theorem}\label{section::debr}
In this section, we state de Branges's theorem which is significant in the proof of Theorem \ref{thm::conformalvariance}.
\begin{thm}[de Branges's theorem]\label{theorem::debran}
Let 
\[S:=\{\phi:\mathbb{D}\rightarrow\mathbb{C} \textrm{ is analytic and one-to-one on the unit disk}, \phi(0)=0, \phi'(0)=1\}\] 
be the set of schlicht functions. For any $\phi\in S$, we consider the Taylor expansion at 0,
\[f(z)=z+\sum_{n=2}^{\infty}a_{n}z^{n}.\]
Then we have $|a_{n}|\le n$ for all $n\ge 2$.
\end{thm}

\section{Proofs of the main results}\label{section::proof}
\subsection{Proof of Theorem \ref{theorem::measureconvergence}}\label{section::circ}
First, we generalize the convergence of $\mu_{\epsilon}$ from along negative powers of $2$ to negative powers of $2^{\frac{1}{N}}$ for any positive integer $N$.
\begin{lem}\label{lemma::Nconvergence}
Fix $\gamma\in[0,2)$ and define $D, h,\mu_{\epsilon},\mu$ as in Section \ref{section::definition}. Then it is almost surely the case that for each positive integer $N$, as $\epsilon\rightarrow 0$ along negative powers of $2^{\frac{1}{N}}$, the measures $\mu_{\epsilon}$ converge weakly in $D$ to $\mu$.
\end{lem}

\begin{proof}
It suffices to show that for any integer $N\ge 2$ and $n\in\{1,2,\ldots, N-1\}$, it is almost surely the case that the measures $\mu_{2^{-\frac{n}{N}-k}}$ converge weakly in $D$ to $\mu$ as $k\rightarrow\infty$. We will use the same method as in \cite[Section 3.2]{DuplantierSheffieldLQGKPZ}. It is easy to see that if for each dyadic square $S$ compactly supported on $D$, the random variables $\mu_{2^{-\frac{n}{N}-k}}(S)$ a.s.\ converge to $\mu(S)$ as $k\rightarrow\infty$, then the desired result follows. Without loss of generality, we assume that $S$ is the unit square $[0,1]^{2}$.

For $y=(y_{1}, y_{2})\in (0,1)^{2}$ and $k\ge 1$, let $S_{k}^{y}$ be the discrete set of $2^{2k}$ points $(a,b)\in S$ with the property that $(2^{k}a-2^{k}y_{1}, 2^{k}b-2^{k}y_{2})\in\mathbb{Z}^{2}$. Define
\[A_{k}^{y}:=2^{-2k}\sum_{z\in S_{k}^{y}}\exp \bar{h}_{2^{-k-1}}(z), \quad B_{k}^{y}:=2^{-2k}\sum_{z\in S_{k}^{y}}\exp \bar{h}_{2^{-\frac{n}{N}-k-1}}(z).\]
Assume that $k$ is large enough such that $\textrm{dist}(S,\partial D)>2^{-k-1}$. By the Markov property of the GFF  and Proposition \ref{prop::circleaverage}, we have that conditioned on the values of $h_{2^{-k-1}}(z)$ for $z\in S_{k}^{y}$, the random variables $h_{2^{-\frac{n}{N}-k-1}}(z)$ are independent of one another and each is a Gaussian random variable with mean $h_{2^{-k-1}}(z)$ and variance $\frac{n}{N}\log 2$. Hence, given the values of $h_{2^{-k-1}}(z)$ for $z\in S_{k}^{y}$, the conditional expectation of $\left|A_{k}^{y}-B_{k}^{y}\right|^{2}$ is 
\begin{equation}\label{equation::uncon}
\mathbb{E}\left(\left.\left|A_{k}^{y}-B_{k}^{y}\right|^{2}\right|h_{2^{-k-1}}(z), z\in S_{k}^{y}\right)=2^{-4k}\tilde{C}\sum_{z\in S_{k}^{y}}e^{2\bar{h}_{2^{-k-1}}(z)},
\end{equation}
where 
\[\tilde{C}=\mathbb{E}\left(\left.\left|1-2^{-\frac{n\gamma^{2}}{2N}}\exp\left[\gamma h_{2^{-\frac{n}{N}-k-1}}(z)\right]\right|^{2}\right|h_{2^{-k-1}}(z)=0\right)=e^{\frac{n\gamma^{2}}{N}}-1\]
is a constant independent of $k$ and $z$.

And the unconditional expectation is\footnote{For $f,g>0$, we write $f\asymp g$ if there exists a constant $c\ge 1$ such that $c^{-1}f(x)\le g(x)\le cf(x)$ for all $x$. We write $f\lesssim g$ if there exists a constant $c>0$ such that $f(x)\le cg(x)$ and $f\gtrsim g$ if $g\lesssim f$.}
\[\mathbb{E}\left|A_{k}^{y}-B_{k}^{y}\right|^{2}=\tilde{C}2^{-4k}2^{(k+1)\gamma^{2}}\sum_{z\in S_{k}^{y}}C(z;D)^{2\gamma^{2}}\asymp 2^{(\gamma^{2}-2)k},\]
where the constant in $\asymp$ is independent of $k$ and $y$.

Note that $\mu_{2^{-k-1}}(S)$ is the mean value of $A_{k}^{y}$ over $y\in[0,1]^{2}$ and $\mu_{2^{-\frac{n}{N}-k-1}}(S)$ is the mean value of $B_{k}^{y}$ over $y\in[0,1]^{2}$. By Jensen's inequality, when $0\le\gamma^{2}<2$, $\mathbb{E}\left|\mu_{2^{-k-1}}(S)-\mu_{2^{-\frac{n}{N}-k-1}}(S)\right|^{2}$ decays exponentially in $k$. The desired result thus follows from the Borel-Cantelli lemma and the almost sure convergence of $\mu_{2^{-k-1}}(S)$ to $\mu(S)$ as $k\rightarrow\infty$.

The case for $\sqrt{2}\le\gamma<2$ can be proved as in \cite[Section 3.2]{DuplantierSheffieldLQGKPZ} by breaking the sum over $z\in S^{y}_{k}$ into two parts. To be self-contained, we include the proof in this paper. Fix $\alpha\in(\gamma,2\gamma)$. Let 
\begin{equation}\label{equation::tildes}
\tilde{S}_{k}^{y}:=\left\{z\in S_{k}^{y}: h_{2^{-k-1}}(z)>-\alpha\log\left(2^{-k-1}/C(z;D)\right)\right\}.
\end{equation}
Define
\[\tilde{A}_{k}^{y}:=2^{-2k}\sum_{z\in S_{k}^{y}}1_{\{z\in\tilde{S}_{k}^{y}\}}\exp \bar{h}_{2^{-k-1}}(z),\quad\tilde{B}_{k}^{y}:=2^{-2k}\sum_{z\in S_{k}^{y}}1_{\{z\in\tilde{S}_{k}^{y}\}}\exp \bar{h}_{2^{-\frac{n}{N}-k-1}}(z).\]
We claim that $\mathbb{E}\tilde{A}_{k}$ and $\mathbb{E}\tilde{B}_{k}$ converge to 0 exponentially in $k$. By Proposition \ref{prop::circleaverage}, the random variable $h_{2^{-k-1}}(z)$ is a centered Gaussian with variance 
\[\sigma^{2}=-\log\left(2^{-k-1}/C(z;D)\right).\] 
Therefore
\begin{eqnarray*}
\mathbb{E}1_{\{z\in\tilde{S}_{k}^{y}\}}\exp\bar{h}_{2^{-k-1}}(z)&=&\frac{2^{-(k+1)\gamma^{2}/2}}{\sqrt{2\pi}\sigma}\int_{\alpha\sigma^{2}}^{\infty}e^{\gamma\eta}e^{-\frac{\eta^{2}}{2\sigma^{2}}}d\eta\\
&=&\frac{\int_{\alpha\sigma^{2}}^{\infty}e^{\gamma\eta}e^{-\frac{\eta^{2}}{2\sigma^{2}}}d\eta}{\int_{-\infty}^{\infty}e^{\gamma\eta}e^{-\frac{\eta^{2}}{2\sigma^{2}}}d\eta}\mathbb{E}e^{\bar{h}_{\epsilon}(z)}\\
&=&\mathbb{P}(X>\alpha\sigma^{2})C(z;D)^{\frac{\gamma^{2}}{2}}\quad (\textrm{random variable }X\sim \mathscr{N}(\gamma\sigma^{2},\sigma^{2}))\\
&\lesssim&2^{-\frac{(\alpha-\gamma)^{2}}{2}k}.\\
\end{eqnarray*}
We also have
\begin{eqnarray*}
\mathbb{E}1_{\{z\in\tilde{S}_{k}^{y}\}}\exp\bar{h}_{2^{-\frac{n}{N}-k-1}}(z)&=&\mathbb{E}\left[1_{\{z\in\tilde{S}_{k}^{y}\}}\mathbb{E}\left[\exp\bar{h}_{2^{-\frac{n}{N}-k-1}}(z)\left|h_{2^{-k-1}}(z)\right.\right]\right]\\
&=&\mathbb{E}1_{\{z\in\tilde{S}_{k}^{y}\}}\exp\bar{h}_{2^{-k-1}}(z).\\
\end{eqnarray*}
Therefore $\mathbb{E}\tilde{B}_{k}^{y}=\mathbb{E}\tilde{A}_{k}^{y}\lesssim 2^{-\frac{(\alpha-\gamma)^{2}}{2}k}$, where the constant in $\lesssim$ is independent of $k$ and $y$. Moreover, it follows from the argument before (\ref{equation::uncon}) that
\begin{eqnarray*}
\lefteqn{\mathbb{E}|(B_{k}^{y}-\tilde{B}_{k}^{y})-(A_{k}^{y}-\tilde{A}_{k}^{y})|^{2}}\\
&=&2^{-4k}\mathbb{E}\left|\sum_{z\in S_{k}^{y}\backslash\tilde{S}_{k}^{y}}\exp \bar{h}_{2^{-k-1}}(z)-\exp \bar{h}_{2^{-\frac{n}{N}-k-1}}(z)\right|^{2}\\
&=&2^{-4k}\tilde{C}\sum_{z\in S_{k}^{y}}\mathbb{E}1_{\{z\in S_{k}^{y}\backslash\tilde{S}_{k}^{y}\}}e^{2\bar{h}_{2^{-k-1}}(z)}\\
&=&2^{-4k}\tilde{C}\sum_{z\in S_{k}^{y}}\frac{\int_{-\infty}^{\alpha\sigma^{2}}e^{2\gamma\eta}e^{-\frac{\eta^{2}}{2\sigma^{2}}}d\eta}{\int_{-\infty}^{\infty}e^{2\gamma\eta}e^{-\frac{\eta^{2}}{2\sigma^{2}}}d\eta}\mathbb{E}e^{2\bar{h}_{2^{-k-1}}(z)}\\
&\lesssim&2^{-(2-\gamma^{2})k}\mathbb{P}(X<\alpha\sigma^{2})\quad(\textrm{random variable }X\sim \mathscr{N}(2\gamma\sigma^{2},\sigma^{2}))\\
&\lesssim&2^{-\left(2-\gamma^{2}+(2\gamma-\alpha)^{2}/2\right)k},\\
\end{eqnarray*}
where the constant in $\lesssim$ is independent of $k$ and $y$.

Note that for a given $\gamma<2$, when $\alpha\in(\gamma,2\gamma)$ is very close to $\gamma$, the exponent becomes close to $2-\frac{\gamma^{2}}{2}>0$. Therefore, we can choose $\alpha$ small enough to make the exponent positive. Therefore the Cauchy-Schwarz inequality and the triangle inequality imply that $|\mathbb{E}A_{k}^{y}-\mathbb{E}B_{k}^{y}|$ decays to 0 exponentially in $k$ and the desired result follows.
\end{proof}
Now we prove Theorem \ref{theorem::measureconvergence}.
\begin{proof}[Proof of Theorem \ref{theorem::measureconvergence}]
For each positive integer $N$, define 
\[\overline{C}(N):=\left(\mathbb{E}\sup_{t\in[0,\frac{\log 2}{N}]}e^{-\frac{\gamma^{2}t}{2}}e^{\gamma B_{t}}\right)^{-1},\]
and
\[\underline{C}(N):=\left(\mathbb{E}\inf_{t\in[0,\frac{\log 2}{N}]}e^{-\frac{\gamma^{2}t}{2}}e^{\gamma B_{t}}\right)^{-1},\]
where $B_{t}$ is a standard one dimensional Brownian motion with $B_{0}=0$.

Moreover, given $N$, for each integer $k\ge 1$, define the random measures
\[\overline{\mu}_{k,N}=\overline{C}(N)\sup_{\epsilon\in[2^{-(k+1)/N},2^{-k/N}]}e^{\bar{h}_{\epsilon}(z)}dz,\]
and
\[\underline{\mu}_{k,N}=\underline{C}(N)\inf_{\epsilon\in[2^{-(k+1)/N},2^{-k/N}]}e^{\bar{h}_{\epsilon}(z)}dz.\]
From Proposition \ref{prop::circleaverage}, it is easy to see that the value of $\overline{C}(N)$ is chosen so that if $B_{2^{-k/N}}(z)\subset D$, we have
\begin{eqnarray*}
\lefteqn{\mathbb{E}\left(\left.\overline{C}(N)\sup_{\epsilon\in[2^{-(k+1)/N},2^{-k/N}]}e^{\bar{h}_{\epsilon}(z)}\right|h_{2^{-k/N}}(z)\right)}\\
&=&e^{\bar{h}_{2^{-k/N}}(z)}\overline{C}(N)\mathbb{E}\left(\left.\sup_{\epsilon\in[2^{-(k+1)/N},2^{-k/N}]}e^{\bar{h}_{\epsilon}(z)-\bar{h}_{2^{-k/N}}(z)}\right|h_{2^{-k/N}}(z)\right)\\
&=&e^{\bar{h}_{2^{-k/N}}(z)}\overline{C}(N)\mathbb{E}\left(\sup_{t\in[0,\frac{\log 2}{N}]}e^{-\frac{\gamma^{2}t}{2}}e^{\gamma B_{t}}\right)\\
&=&e^{\bar{h}_{2^{-k/N}}(z)}.\\
\end{eqnarray*}
Therefore, by estimating $\mathbb{E}\left|\mu_{2^{-k/N}}(S)-\overline{\mu}_{k,N}(S)\right|^{2}$ using the tower property of conditional expectation as in the proof of Lemma $\ref{lemma::Nconvergence}$, it is easy to see that for each $N$, it is almost surely the case that the measures $\overline{\mu}_{k,N}$ converge weakly to $\mu$ as $k\rightarrow\infty$. Similarly, we can prove that the same result holds for $\underline{\mu}_{k, N}$. The monotone convergence theorem implies that both $\overline{C}(N)$ and $\underline{C}(N)$ converge to 1 as $N\rightarrow\infty$. The desired result thus follows.
\end{proof}

\subsection{Proof of Theorem \ref{prop::convergence}}\label{section::bump}
We first recall the setup described in Section \ref{section::definition}. Suppose that $f:\mathbb{R}^{2}\rightarrow\mathbb{R}_{\ge 0}$ is a radially symmetric bump function compactly supported on $B_{1}(0)$ with \[\int_{B_{1}(0)}f(z)dz=1.\]
For $0<\epsilon<1$, define 
\begin{equation}\label{equation::bump}
f_{\epsilon}(z):=\frac{1}{\epsilon^{2}}f\left(\frac{z}{\epsilon}\right).
\end{equation}
We occasionally abuse notation and by writing $f(r)$ for $r\ge 0$ we mean $f((r,0))$. Then the convolution of the GFF with $f_{\epsilon}$ becomes
\[h\ast f_{\epsilon}(z)=(h, f_{\epsilon}(z-\cdot))=2\pi\int_{0}^{\epsilon}h_{r}(z)f_{\epsilon}(r)rdr.\]
If $B_{\epsilon}(z)\subset D$, it follows from (\ref{equation::variance}) that $h\ast f_{\epsilon}(z)$ is a Gaussian random variable with mean zero and variance
\begin{eqnarray*}
\textrm{var}\left(h\ast f_{\epsilon}(z)\right)&=&4\pi^{2}\iint_{[0,\epsilon]^{2}}\mathbb{E}\left(h_{x}(z)h_{y}(z)\right)f_{\epsilon}(x)f_{\epsilon}(y)xydxdy\\
&=&4\pi^{2}\iint_{[0,\epsilon]^{2}}\left[-\log(x\vee y)+\log C(z;D)\right]f_{\epsilon}(x)f_{\epsilon}(y)xydxdy\\
&=&-C-\log\epsilon+\log C(z;D),\\
\end{eqnarray*}
where 
\begin{equation}\label{equation::defC}
C=4\pi^{2}\iint_{[0,1]^{2}}\log(x\vee y)f(x)f(y)xydxdy<0
\end{equation}
is a constant.

Throughout the remainder of this article, we fix the bump function $f$ and define
\[\tilde{h}_{\epsilon}(z):=\gamma h\ast f_{\epsilon}(z)+\frac{\gamma^{2}}{2}(\log\epsilon+C),\]
and
\begin{equation*}
\tilde{\mu}_{\epsilon}:=e^{\tilde{h}_{\epsilon}(z)}dz.
\end{equation*}
It is easy to check that the definition of $\tilde{\mu}_{\epsilon}$ as above coincides with the one given by (\ref{equation::approximate}).
\begin{proof}[Proof of Theorem \ref{prop::convergence}]
Note that if $B_{\epsilon}(z)\subset D$, we have
\[\mathbb{E}\left(\left.\exp\left(\tilde{h}_{\epsilon}(z)\right)\right|h_{\epsilon}(z)\right)=\exp\left(\bar{h}_{\epsilon}(z)\right),\]
and  
\[\mathbb{E}\left(\left.\left|\exp\left(\tilde{h}_{\epsilon}(z)\right)-\exp\left(\bar{h}_{\epsilon}(z)\right)\right|^{2}\right|h_{\epsilon}(z)\right)=\tilde{C}\exp\left(2\bar{h}_{\epsilon}(z)\right),\]
where 
\[\tilde{C}=\mathbb{E}\left(\left.\left|1-e^{\frac{C\gamma^{2}}{2}+\gamma h*f_{\epsilon}(z)}\right|^{2}\right|h_{\epsilon}(z)=0\right)=e^{-C\gamma^{2}}-1\]
is a constant independent of $\epsilon$ and $z$.

For each positive integer $N$, by estimating $\mathbb{E}\left|\tilde{\mu}_{2^{-k/N}}(S)-\mu_{2^{-k/N}}(S)\right|^{2}$ as in the proof of Lemma \ref{lemma::Nconvergence}, one can show that it is almost surely the case that the measures $\tilde{\mu}_{2^{-k/N}}$ converge weakly in $D$ to $\mu$ as $k\rightarrow\infty$.

Let 
\[\tilde{C}(N):=\frac{\mathbb{E}\exp\left[2\pi\gamma\int_{0}^{\infty}B_{t}f(e^{-t})e^{-2t}dt\right]}{\mathbb{E}\exp\left[2\pi\gamma\sup_{\Delta\in[0,\frac{\log 2}{N}]}\int_{0}^{\infty}\left(B_{t+\Delta}-\frac{\gamma\Delta}{2}\right)f(e^{-t})e^{-2t}dt\right]},\]
and
\[\uwave{C}(N):=\frac{\mathbb{E}\exp\left[2\pi\gamma\int_{0}^{\infty}B_{t}f(e^{-t})e^{-2t}dt\right]}{\mathbb{E}\exp\left[2\pi\gamma\inf_{\Delta\in[0,\frac{\log 2}{N}]}\int_{0}^{\infty}\left(B_{t+\Delta}-\frac{\gamma\Delta}{2}\right)f(e^{-t})e^{-2t}dt\right]},\]
where $B_{t}$ is a standard one dimensional Brownian motion with $B_{0}=0$.

From Proposition \ref{prop::circleaverage}, it is easy to see that the values of $\tilde{C}(N)$ and $\uwave{C}(N)$ are chosen so that if $B_{2^{-k/N}}(z)\subset D$, we have
\[\mathbb{E}\left(\left.\tilde{C}(N)\sup_{\epsilon\in[2^{-(k+1)/N},2^{-k/N}]}\exp\left(\tilde{h}_{\epsilon}(z)\right)\right|h_{2^{-k/N}}(z)\right)=\exp\left(\bar{h}_{2^{-k/N}}(z)\right),\]
and
\[\mathbb{E}\left(\left.\uwave{C}(N)\inf_{\epsilon\in[2^{-(k+1)/N},2^{-k/N}]}\exp\left(\tilde{h}_{\epsilon}(z)\right)\right|h_{2^{-k/N}}(z)\right)=\exp\left(\bar{h}_{2^{-k/N}}(z)\right).\]

Therefore, as in the proof of Theorem \ref{theorem::measureconvergence}, one can show that for each positive integer $N$, it is almost surely the case that the following two sequences of random measures
\[\tilde{C}(N)\sup_{\epsilon\in[2^{-(k+1)/N},2^{-k/N}]}\exp\left(\tilde{h}_{\epsilon}(z)\right)dz\]
and
\[\uwave{C}(N)\inf_{\epsilon\in[2^{-(k+1)/N},2^{-k/N}]}\exp\left(\tilde{h}_{\epsilon}(z)\right)dz\]
both converge weakly to $\mu$ as $k\rightarrow\infty$.

It is easy to check that by the dominated convergence theorem, we have that both $\tilde{C}(N)$ and $\uwave{C}(N)$ converge to $1$ as $N\rightarrow\infty$. The desired result thus follows.
\end{proof}
Using a similar argument, we are able to generalize Theorem \ref{prop::convergence} in such a way that in the definition of the approximating measures, the $\epsilon$ in $\tilde{h}_{\epsilon}$ varies continuously with respect to the point $z$:
\begin{cor}\label{lemma::differentradius}
Fix $\gamma\in[0,2)$ and assume the same notations as in Theorem \ref{prop::convergence}. Suppose that $g: D\rightarrow\mathbb{R}_{> 0}$ is a smooth function on $D$. Then it is almost surely the case that as $\epsilon\rightarrow 0$, the measures $e^{\tilde{h}_{\epsilon g(z)}(z)}dz$ converge weakly in $D$ to $\mu$.
\end{cor}
\begin{rk}
Recall that $\Lambda$ is the collection of the conformal maps onto $D$ as defined in (\ref{equation::Lambda}).
For $\phi\in\Lambda$, we have
\[(h\circ\phi)\ast f_{\epsilon}(z)=\left(h, |(\phi^{-1})'|^{2}f_{\epsilon}\left(z-\phi^{-1}(\cdot)\right)\right).\]
Since $h$ is almost surely a distribution on $D$, it is almost surely the case that $(h\circ\phi)\ast f_{\epsilon}(z)$ are well-defined simultaneously for all $\phi\in\Lambda$, each viewed as an approximation of $h\circ\phi$, which makes it possible for us to prove that the desired results of $\mu^{h\circ\phi}$ in Theorem \ref{thm::conformalvariance} hold simultaneously for all $\phi\in\Lambda$.
\end{rk}
In summary, to define the Liouville quantum gravity measure $\mu^{h}=e^{\gamma h(z)}dz$, we start by approximate the GFF $h$ by the circle average process $h_{\epsilon}(z)$. It is known that the approximating measure a.s. weakly converges to a limiting measure along the geometric progression $\epsilon_{k}=2^{-k}$ (\cite{DuplantierSheffieldLQGKPZ}). We thus define $\mu^{h}$ the to be this limiting measure. In Theorem \ref{theorem::measureconvergence}, we generalize the a.s.\ weak convergence of the approximating measure from along $\epsilon_{k}$ to the situation where $\epsilon\rightarrow 0$ continuously. The measure $\mu^{h}$ is also the a.s.\ limit if we approximate $h$ by the convolution with some mollifier $f$ (Theorem \ref{prop::convergence}). Moreover, we will obtain $\mu^{h}$ as the limiting measure even when the scale $\epsilon=\epsilon(z)$ varies continuously with respect to $z$ (Corollary \ref{lemma::differentradius}).

\subsection{Proof of Theorem \ref{thm::conformalvariance}}\label{section::transform}
Now we can show that it is almost surely the case that the transformation rule (\ref{equation::transformrule}) holds simultaneously for all $\phi\in\Lambda$.

\begin{proof}[Proof of Theorem \ref{thm::conformalvariance}]
For each $\phi\in\Lambda$, in Corollary \ref{lemma::differentradius}, we set $g(z)=1/|\phi'(z)|$ and let $\epsilon$ converge to 0 along the sequence $\{2^{-k}, k\in\mathbb{N}\}$. We thus obtain that $\mu^{h\circ\phi+Q\log|\phi'|}$ is a.s. the weak limit of  \[\left(\frac{e^{C}\epsilon}{|\phi'(\omega)|}\right)^{\frac{\gamma^{2}}{2}}\exp\left[\gamma(h\circ\phi+Q\log|\phi'|)*f_{\frac{2^{-k}}{|\phi'(\omega)|}}(\omega)\right]d\omega\] 
as $k\rightarrow\infty$. It suffices to show that for each dyadic square $S$ compactly supported on $D$, it is almost surely the case that as $\epsilon\rightarrow 0$ along negative powers of 2, we have 
\begin{equation}\label{equation::coordinate}
\int_{\phi^{-1}(S)}\left(\frac{e^{C}\epsilon}{|\phi'(\omega)|}\right)^{\frac{\gamma^{2}}{2}}\exp\left[\gamma(h\circ\phi+Q\log|\phi'|)*f_{\frac{\epsilon}{|\phi'(\omega)|}}(\omega)\right]d\omega
\end{equation}
converge to $\mu^{h}(S)$ uniformly over $\phi\in\Lambda$, where $C$ is as defined in (\ref{equation::defC}). Once we are able to establish this, it follows immediately that it is almost surely the case that for all $\phi\in\Lambda$, the measure $\mu^{h\circ\phi+Q\log|\phi'|}$ are well-defined and the transformation rule (\ref{equation::transformrule}) holds simultaneously for all $\phi\in\Lambda$.

Change of coordinates implies that
\begin{eqnarray*}
\lefteqn{\int_{\phi^{-1}(S)}\left(\frac{e^{C}\epsilon}{|\phi'(\omega)|}\right)^{\frac{\gamma^{2}}{2}}\exp\left[\gamma(h\circ\phi+Q\log|\phi'|)*f_{\frac{\epsilon}{|\phi'(\omega)|}}(\omega)\right]d\omega}\\
&=&\int_{\phi^{-1}(S)}\left(\frac{e^{C}\epsilon}{|\phi'(\omega)|}\right)^{\frac{\gamma^{2}}{2}}\exp\left[\gamma(h\circ\phi)*f_{\frac{\epsilon}{|\phi'(\omega)|}}(\omega)+\left(2+\frac{\gamma^{2}}{2}\right)\log|\phi'(\omega)|\right]d\omega\\
&=&\int_{\phi^{-1}(S)}(e^{C}\epsilon)^{\frac{\gamma^{2}}{2}}\exp\left[\gamma(h\circ\phi)*f_{\frac{\epsilon}{|\phi'(\omega)|}}(\omega)\right]|\phi'(\omega)|^{2}d\omega\\
&=&\int_{S}(e^{C}\epsilon)^{\frac{\gamma^{2}}{2}}\exp\left[\gamma(h\circ\phi)*f_{\epsilon|(\phi^{-1})'(z)|}(\phi^{-1}(z))\right]dz.\\
\end{eqnarray*}
Note that when $\phi$ is the identity map on $D$, (\ref{equation::coordinate}) becomes $\tilde{\mu}^{h}_{\epsilon}(S)$. Therefore, Theorem \ref{prop::convergence} implies that the uniform convergence is equivalent to the fact that as $\epsilon\rightarrow 0$ along negative powers of 2, we have
\begin{equation}\label{equation::sup}
(e^{C}\epsilon)^{\frac{\gamma^{2}}{2}}\left|\sup_{\phi\in\Lambda}\int_{S}\exp\left[\gamma(h\circ\phi)*f_{\epsilon|(\phi^{-1})'(z)|}(\phi^{-1}(z))\right]-\exp\left[\gamma h*f_{\epsilon}(z)\right]dz\right|\rightarrow 0,
\end{equation}
and
\begin{equation}\label{equation::inf}
(e^{C}\epsilon)^{\frac{\gamma^{2}}{2}}\left|\inf_{\phi\in\Lambda}\int_{S}\exp\left[\gamma(h\circ\phi)*f_{\epsilon|(\phi^{-1})'(z)|}(\phi^{-1}(z))\right]-\exp[\gamma h*f_{\epsilon}(z)]dz\right|\rightarrow 0
\end{equation}
almost surely.

Without loss of generality, we assume that $S$ is the unit square $[0,1]^{2}$. Suppose that $\epsilon=2^{-k-3}$ for some $k\in\mathbb{N}$. For $y\in (0,1)^{2}$, define the set of points $S_{k}^{y}$ as in the proof of Lemma \ref{lemma::Nconvergence}. Let
\[A_{k}^{y}:=2^{-2k}\sum_{z\in S_{k}^{y}}(e^{C}\epsilon)^{\frac{\gamma^{2}}{2}}\sup_{\phi\in\Lambda}\exp\left[\gamma(h\circ\phi)*f_{\epsilon|(\phi^{-1})'(z)|}(\phi^{-1}(z))\right],\]
and
\[B_{k}^{y}:=2^{-2k}\sum_{z\in S_{k}^{y}}(e^{C}\epsilon)^{\frac{\gamma^{2}}{2}}\exp\left[\gamma h*f_{\epsilon}(z)\right].\]
In order to prove (\ref{equation::sup}), it suffices to show that $\mathbb{E}\left|A_{k}^{y}-B_{k}^{y}\right|^{2}$ decays exponentially in $k$ and uniformly over $y\in (0,1)^{2}.$

Note that 
\[(h\circ\phi)*f_{\epsilon|(\phi^{-1})'(z)|}(\phi^{-1}(z))=\left(h, |(\phi^{-1})'|^{2}f_{\epsilon|(\phi^{-1})'(z)|}(\phi^{-1}(z)-\phi^{-1}(\cdot))\right).\] 
Assume that $k$ is large enough such that $\textrm{dist}(S,\partial D)>2^{-k-1}$. The Koebe 1/4 theorem implies that $\phi\left(B_{\epsilon|(\phi^{-1})'(z)|}(\phi^{-1}(z))\right)\subset B_{4\epsilon}(z)$. Recall that $\epsilon=2^{-k-3}$ and thus the sets in $\{B_{4\epsilon}(z), z\in S_{k}^{y}\}$ are disjoint. By the Markov property of the GFF and Remark \ref{remark::independence}, we have that conditioned on the values of $h_{2^{-k-1}}(z)$ for $z\in S_{k}^{y}$, the random variables \[\sup_{\phi\in\Lambda}\exp\left[\gamma(h\circ\phi)*f_{\epsilon|(\phi^{-1})'(z)|}(\phi^{-1}(z))\right]-\exp[\gamma h*f_{\epsilon}(z)]\]
are independent of one another and each has the conditional law as that of 
\begin{equation}\label{equation::term}
e^{\gamma h_{2^{-k-1}}(z)}\left(\sup_{\phi\in\Lambda}\exp\left[\gamma(h^{z}\circ\phi)*f_{\epsilon|\left(\phi^{-1}\right)'(z)|}(\phi^{-1}(z))\right]-\exp[\gamma h^{z}*f_{\epsilon}(z)]\right),
\end{equation}
where $h^{z}$ is an instance of the GFF on $B_{4\epsilon}(z)$, independent of $h_{2^{-k-1}}(z)$.

In the following, we will give an estimate of the second moment of (\ref{equation::term}). For the ease of notation we assume that $z=0\in D$, $\phi^{-1}(0)=0$ and $\textrm{dist}(0,\partial D)>4\epsilon$. Then we have
\begin{eqnarray*}
\lefteqn{\sup_{\phi\in\Lambda}\exp\left[\gamma (h^{0}\circ\phi)*f_{\epsilon|(\phi^{-1})'(0)|}(0)\right]-\exp\left[\gamma h^{0}*f_{\epsilon}(0)\right]}\\
&=&\sup_{\phi\in\Lambda}\exp\left[\gamma\left(h^{0}, \left|\frac{(\phi^{-1})'(\cdot)}{\epsilon(\phi^{-1})'(0)}\right|^{2}f\left(\frac{\phi^{-1}(\cdot)}{\epsilon|(\phi^{-1})'(0)|}\right)\right)\right]-\exp\left[\gamma\left(h^{0},\frac{1}{\epsilon^{2}}f\left(\frac{\cdot}{\epsilon}\right)\right)\right]\\
&=&\sup_{\phi\in\Lambda^{*}}\exp\left[\gamma\left(h^{0}, \left|\frac{(\phi^{-1})'(\cdot)}{\epsilon}\right|^{2}f\left(\frac{\phi^{-1}(\cdot)}{\epsilon}\right)\right)\right]-\exp\left[\gamma\left(h^{0},\frac{1}{\epsilon^{2}}f\left(\frac{\cdot}{\epsilon}\right)\right)\right]\\
&\stackrel{d}{=}&\sup_{\phi\in\Lambda^{*}}\exp\left[\gamma\left(h^{*}, |(\phi^{-1})'(\epsilon\cdot)|^{2}f\left(\frac{\phi^{-1}(\epsilon\cdot)}{\epsilon}\right)\right)\right]-\exp\left[\gamma(h^{*},f)\right],\\
\end{eqnarray*}
where 
\begin{equation}\label{equation::lambdastar}
\Lambda^{*}=\{\phi\in\Lambda: \phi^{-1}(0)=0, \left(\phi^{-1}\right)'(0)=1\},
\end{equation}
and $h^{*}$ is an instance of the GFF on $B_{4}(0)$. The second equality follows from the fact that by considering $\frac{\phi^{-1}(\cdot)}{|(\phi^{-1})'(0)|}$ for each $\phi\in\Lambda$ and the radial symmetry of $f$, it suffices to take the supremum over $\phi\in\Lambda^{*}$. The third equality follows from the conformal invariance of the GFF. 

For each $\phi\in\Lambda^{*}$, we define the function $f_{\epsilon}^{\phi}: B_{4}(0)\rightarrow\mathbb{R}$ as
\begin{equation}\label{equation::fdef} f^{\phi}_{\epsilon}(\cdot):=|(\phi^{-1})'(\epsilon\cdot)|^{2}f\left(\frac{\phi^{-1}(\epsilon\cdot)}{\epsilon}\right).
\end{equation}
It is easy to see that $f^{\phi}_{\epsilon}$ is a $C^{\infty}$ real-valued function compactly supported on $B_{4}(0)$.

We claim that $m_{\epsilon}:=\mathbb{E}\sup_{\phi\in\Lambda^{*}}(h^{*},f_{\epsilon}^{\phi})$ converges to 0 as $\epsilon\rightarrow 0$ and for any $\alpha>0$, we have
\begin{equation}\label{equation::est}
0\le\mathbb{E}e^{\alpha\left(\sup_{\phi\in\Lambda^{*}}(h^{*}, f_{\epsilon}^{\phi})-(h^{*},f)\right)}-1\lesssim m_{\epsilon}\vee\epsilon^{2},
\end{equation}
where the constant in $\lesssim$ only depends on $\alpha$. We leave the proof of this claim to Lemma \ref{lemma::unicon} and Lemma \ref{lemma::est}.

By H\"{o}lder's inequality and (\ref{equation::est}), we have the $L^{1}$ estimate:
\begin{equation}\label{equation::L1est}
\begin{split}
&\mathbb{E}\left|e^{\gamma\sup_{\phi\in\Lambda^{*}}(h^{*}, f_{\epsilon}^{\phi})}-e^{\gamma(h^{*},f)}\right|\\
&\le\left(\mathbb{E}e^{2\gamma(h^{*},f)}\mathbb{E}\left|e^{\gamma\left[\sup_{\phi\in\Lambda^{*}}(h^{*}, f_{\epsilon}^{\phi})-(h^{*},f)\right]}-1\right|^{2}\right)^{\frac{1}{2}}\\
&\lesssim\left(\mathbb{E}e^{2\gamma\left[\sup_{\phi\in\Lambda^{*}}(h^{*}, f_{\epsilon}^{\phi})-(h^{*},f)\right]}-2\mathbb{E}e^{\gamma\left[\sup_{\phi\in\Lambda^{*}}(h^{*}, f_{\epsilon}^{\phi})-(h^{*},f)\right]}+1\right)^{\frac{1}{2}}\\
&\le\left(\mathbb{E}e^{2\gamma\left[\sup_{\phi\in\Lambda^{*}}(h^{*}, f_{\epsilon}^{\phi})-(h^{*},f)\right]}-1\right)^{\frac{1}{2}}\\
&\lesssim\sqrt{m_{\epsilon}}\vee\epsilon,\\
\end{split}
\end{equation}
where the constants in $\lesssim$ only depend on $\gamma$.

Let 
\[\overline{C}_{\epsilon}:=\frac{\mathbb{E}e^{\gamma (h^{*},f)}}{\mathbb{E}e^{\gamma\sup_{\phi\in\Lambda^{*}}(h^{*}, f_{\epsilon}^{\phi})}}.\]
Then (\ref{equation::L1est}) implies that 
\begin{equation}\label{equation::estC}
0\le 1-\overline{C}_{\epsilon}\lesssim\sqrt{m_{\epsilon}}\vee\epsilon,
\end{equation} 
where the constant in $\lesssim$ only depends on $\gamma$.

By (\ref{equation::est}), (\ref{equation::estC}) and H\"{o}lder's inequality, we can also obtain the $L^{2}$ estimate:
\begin{equation}\label{equation::L2est}
\begin{split}
&\mathbb{E}\left|\overline{C}_{\epsilon}e^{\gamma\sup_{\phi\in\Lambda^{*}}(h^{*}, f_{\epsilon}^{\phi})}-e^{\gamma (h^{*},f)}\right|^{2}\\
&\le 2\overline{C}_{\epsilon}^{2}\left(\mathbb{E}e^{4\gamma(h^{*},f)}\mathbb{E}\left|e^{\gamma\left[\sup_{\phi\in\Lambda^{*}}(h^{*}, f_{\epsilon}^{\phi})-(h^{*},f)\right]}-1\right|^{4}\right)^{\frac{1}{2}}+2(\overline{C}_{\epsilon}-1)^{2}\mathbb{E}e^{2\gamma(h^{*},f)}\\
&\lesssim \sqrt{m_{\epsilon}}\vee\epsilon,\\
\end{split}
\end{equation}
where the constant in $\lesssim$ only depends on $\gamma$.

For each $z\in S^{y}_{k}$, let 
\[C_{\epsilon}(z):=\frac{\mathbb{E}\exp[\gamma h^{z}*f_{\epsilon}(z)]}{\mathbb{E}\sup_{\phi\in\Lambda}\exp\left[\gamma(h^{z}\circ\phi)*f_{\epsilon|\left(\phi^{-1}\right)'(z)|}(\phi^{-1}(z))\right]}.\]
It is easy to check that $C_{\epsilon}(z)$ has the same bound 
\begin{equation}\label{equation::czest}
0\le 1-C_{\epsilon}(z)\lesssim\sqrt{m_{\epsilon}}\vee\epsilon,
\end{equation} 
where the constant in $\lesssim$ is uniform over $z$, and the above estimates hold when we replace 0 with any $z\in S_{k}^{y}$.
We modify $A_{k}^{y}$ as
\[\hat{A}_{k}^{y}:=2^{-2k}\sum_{z\in S_{k}^{y}}(e^{C}\epsilon)^{\frac{\gamma^{2}}{2}}C_{\epsilon}(z)\sup_{\phi\in\Lambda}\exp\left[\gamma(h\circ\phi)*f_{\epsilon|(\phi^{-1})'(z)|}(\phi^{-1}(z))\right]\]
to make the cross terms vanish in the estimate of $\mathbb{E}\left|\hat{A}_{k}^{y}-B_{k}^{y}\right|^{2}$. As in the proof of Lemma \ref{lemma::Nconvergence}, by H\"{o}lder's inequality and (\ref{equation::L2est}), we have
\begin{equation}\label{equation::squareest}
\mathbb{E}\left|\hat{A}_{k}^{y}-B_{k}^{y}\right|^{2}\lesssim 2^{-4k}\sum_{z\in S_{k}^{y}}\mathbb{E}e^{2\bar{h}_{2^{-k-1}}(z)}\lesssim 2^{(\gamma^{2}-2)k},
\end{equation}
where the constants in $\lesssim$ are independent of $k$ and $y$.

When $0\le\gamma<\sqrt{2}$, combining (\ref{equation::squareest}) with (\ref{equation::czest}) yields (\ref{equation::sup}). The case for $\sqrt{2}\le\gamma<2$ can be proved by breaking the sum over $z\in S^{y}_{k}$ into two parts as in the proof of Lemma \ref{lemma::Nconvergence}. To be more precise, we define $\tilde{S}_{k}^{y}$ as in (\ref{equation::tildes}) and define $\tilde{A}_{k}^{y}$ and $\tilde{B}_{k}^{y}$ based on $\hat{A}_{k}^{y}$ and $B_{k}^{y}$ analogously. The conditional distribution argument before (\ref{equation::term}) and the estimate (\ref{equation::L1est}), (\ref{equation::czest}) and (\ref{equation::L2est}) imply that
\[\mathbb{E}|\tilde{A}_{k}^{y}-\tilde{B}_{k}^{y}|\lesssim\mathbb{E}1_{\{z\in\tilde{S}_{k}^{y}\}}\exp\bar{h}_{2^{-k-1}}(z),\]
and
\[\mathbb{E}|\hat{A}_{k}^{y}-\tilde{A}_{k}^{y}+B_{k}^{y}-\tilde{B}_{k}^{y}|^{2}\lesssim 2^{-4k}\sum_{z\in S_{k}^{y}}\mathbb{E}1_{\{z\in S_{k}^{y}\backslash\tilde{S_{k}^{y}}\}}e^{2\bar{h}_{2^{-k-1}}(z)}.\]
The computation at the end of the proof of Lemma \ref{lemma::Nconvergence} leads to the desired result.

(\ref{equation::inf}) can be proved similarly.
\end{proof}
We finish this section with the proof of (\ref{equation::est}).

For a domain $D$, recall that the space of test functions on $D$ is the space of $C^{\infty}$ real-valued functions compactly supported on $D$, equipped with the topology of uniform convergence of all derivatives.
\begin{lem}\label{lemma::unicon}
Let $\Lambda^{*}$ be as in (\ref{equation::lambdastar}) and $f_{\epsilon}^{\phi}$ as in (\ref{equation::fdef}). Then $f_{\epsilon}^{\phi}$ converge to $f$ in the space of test functions on $B_{4}(0)$ uniformly over $\phi\in\Lambda^{*}$ as $\epsilon\rightarrow 0$. Especially, we have
\begin{equation}\label{equation::infnityest}
||f_{\epsilon}^{\phi}-f||_{\infty}\lesssim\epsilon,
\end{equation}
where the constant in $\lesssim$ is uniform over $\phi\in\Lambda^{*}$
\end{lem}
\begin{proof}
Recall (\ref{equation::fdef}) and note that $f$ is compactly supported on $B_{1}(0)$. It thus suffices to show that  $\frac{\phi(\epsilon\cdot)}{\epsilon}$ converges to the identity map with respect to the topology of uniform convergence of all derivatives on $\overline{B_{4}(0)}$ as $\epsilon\rightarrow 0$.

Without loss of generality, we may assume that $\textrm{dist}(0,\partial D)>1$. Then for each $\phi\in\Lambda^{*}$, we have $\phi^{-1}|_{\mathbb{D}}\in S$ (recall the definition of schlicht functions in Section \ref{section::debr}) and by de Branges's theorem (Theorem \ref{theorem::debran}), it has the Taylor expansion
$\phi^{-1}(z)=z+\sum_{n=2}^{\infty}a_{n}z^{n}$ for $|z|<1$ with $|a_{n}|\le n$.

Recall that the Koebe function $f_{\textrm{Koebe}}$ is defined by
\[f_{\textrm{Koebe}}(z):=\frac{z}{(1-z)^{2}}=\sum_{n=1}^{\infty}nz^{n}\quad{\textrm{for $|z|<1$}}.\]
Since we only consider $|z|\le 4$, when $\epsilon<1/8$, we have
\begin{equation}\label{equation::estphi}
\left|\frac{\phi^{-1}(\epsilon z)}{\epsilon}-z\right|\le\frac{1}{\epsilon}\sum_{n=2}^{\infty}n\epsilon^{n}| z|^{n}=\frac{1}{\epsilon}\left(f_{\textrm{Koebe}}(\epsilon|z|)-\epsilon |z|\right)=\epsilon\frac{2|z|^{2}-\epsilon|z|^{3}}{(1-\epsilon|z|)^{2}}\lesssim\epsilon,
\end{equation}
\begin{equation}\label{equation::estdiffphi}
\left|(\phi^{-1})'(\epsilon z)-1\right|\le\sum_{n=2}^{\infty}n^{2}(\epsilon|z|)^{n-1}= f'_{\textrm{Koebe}}(\epsilon|z|)-1=\epsilon\frac{4|z|-3\epsilon|z|^{2}+\epsilon^{2}|z|^{3}}{(1-\epsilon|z|)^{3}}\lesssim\epsilon,
\end{equation}
and generally for $k\ge 2$,
\[\left|\epsilon^{k-1}(\phi^{-1})^{(k)}(\epsilon z)\right|\le \epsilon^{k-1}f^{(k)}_{\textrm{Koebe}}(\epsilon|z|)\lesssim\epsilon^{k-1},\]
where the constants in $\lesssim$ are uniform over $|z|\le 4$ and $\phi\in\Lambda^{*}$. Therefore, we have that $f_{\epsilon}^{\phi}$ converge to $f$ in the space of test functions on $B_{4}(0)$ uniformly over $\phi\in\Lambda^{*}$ as $\epsilon\rightarrow 0$.

Since $f$ is compactly supported on $B_{1}(0)$ and radially symmetric, (\ref{equation::infnityest}) follows from (\ref{equation::estphi}) and (\ref{equation::estdiffphi}).
\end{proof}
\begin{lem}\label{lemma::est}
Assume the same setting as in the proof of Theorem \ref{thm::conformalvariance}. Then $m_{\epsilon}:=\mathbb{E}\sup_{\phi\in\Lambda^{*}}(h^{*},f_{\epsilon}^{\phi})$ converges to 0 as $\epsilon\rightarrow 0$. Moreover, for any $\alpha>0$, we have
\[0\le\mathbb{E}e^{\alpha\left(\sup_{\phi\in\Lambda^{*}}(h^{*}, f_{\epsilon}^{\phi})-(h^{*},f)\right)}-1\lesssim m_{\epsilon}\vee\epsilon^{2},\]
where the constant in $\lesssim$ only depends on $\alpha$.
\end{lem}
\begin{proof}
Let $G(\cdot, \cdot)$ be the Green's function for the Laplacian on $B_{4}(0)$ as defined in (\ref{equation::green}). For each $\phi\in\Lambda^{*}$ it follows from (\ref{equation::infnityest}) that 
\begin{equation}\label{equation::estvar}
\begin{split}
&\textrm{var}\left((h^{*}, f_{\epsilon}^{\phi})-(h^{*},f)\right)\\
&=\iint_{B_{4}(0)\times B_{4}(0)}G(x,y)\left(f_{\epsilon}^{\phi}(x)-f(x)\right)\left(f_{\epsilon}^{\phi}(y)-f(y)\right)dxdy\\
&\le\iint_{B_{4}(0)\times B_{4}(0)}G(x,y)||f_{\epsilon}^{\phi}-f||_{\infty}^{2}dxdy\\
&\lesssim\epsilon^{2}\iint_{B_{4}(0)\times B_{4}(0)}G(x,y)dxdy\\
&\le c\epsilon^{2}, 
\end{split}
\end{equation}
for some constant $c>0$.

Note that for a fixed $\epsilon$, the set of the distorted bump functions $\left\{f_{\epsilon}^{\phi}: \phi\in\Lambda^{*}\right\}$ is compact in the space of test functions on $B_{4}(0)$ with respect to the topology of uniform convergence of all derivatives. Since $h^{*}$ is almost surely a distribution on $B_{4}(0)$, we have $\sup_{\phi\in\Lambda^{*}}(h^{*}, f_{\epsilon}^{\phi})<\infty$
almost surely. By continuity, it is equal to the supremum taken over a countable dense subset. It thus follows from (\ref{equation::estvar}) and Proposition \ref{propsition::supest} that  
\begin{equation}\label{equation::var}
\textrm{var}\left(\sup_{\phi\in\Lambda^{*}}(h^{*}, f_{\epsilon}^{\phi})-(h^{*},f)\right)\le c\epsilon^{2},
\end{equation}
and for any $r>0$, we have
\begin{equation}\label{equation::tail}
\mathbb{P}\left(\left|\sup_{\phi\in\Lambda^{*}}(h^{*}, f_{\epsilon}^{\phi})-(h^{*},f)-m_{\epsilon}\right|>r\right)\le 2 e^{-\frac{r^{2}}{2c\epsilon^{2}}}.
\end{equation}

From Lemma \ref{lemma::unicon} we have that $f_{\epsilon}^{\phi}$ converge to $f$ in the space of test functions on $B_{4}(0)$ uniformly over $\phi\in\Lambda^{*}$ as $\epsilon\rightarrow 0$. A continuity argument as above implies that
$\sup_{\phi\in\Lambda^{*}}(h^{*}, f_{\epsilon}^{\phi})-(h^{*},f)\rightarrow 0$
almost surely as $\epsilon\rightarrow 0$. Combining this fact with (\ref{equation::var}) implies that 
\begin{equation}\label{equation::estexp}
m_{\epsilon}\rightarrow 0\quad\textrm{as}\quad\epsilon\rightarrow 0.
\end{equation}
It thus follows from (\ref{equation::tail}), (\ref{equation::estexp}) and Proposition \ref{proposition::subgaussian} that for any $\alpha>0$, we have
\begin{eqnarray*}
0&\le&\mathbb{E}e^{\alpha\left(\sup_{\phi\in\Lambda^{*}}(h^{*}, f_{\epsilon}^{\phi})-(h^{*},f)\right)}-1\\
&=&e^{\alpha m_{\epsilon}}\mathbb{E}e^{\alpha\left(\sup_{\phi\in\Lambda^{*}}(h^{*}, f_{\epsilon}^{\phi})-(h^{*},f)-m_{\epsilon}\right)}-1\\
&\le&e^{\alpha m_{\epsilon}+4c\alpha^{2}\epsilon^{2}}-1\\
&\lesssim&m_{\epsilon}\vee\epsilon^{2},\\
\end{eqnarray*}
where the constant in $\lesssim$ only depends on $\alpha$.
\end{proof}

\section{Generalization to boundary measures}\label{section::boundary}
The above results about random measures on $D$ have straightforward analogs for random measures on $\partial D$. We first recall some properties of the boundary semicircle average processes. Suppose that $D$ is a bounded simply connected domain whose boundary contains a linear piece $\underline{\partial D}\subset \partial D\cap\mathbb{R}$ and that $h$ is an instance of the GFF on $D$ with \textit{free} boundary conditions, normalized to have mean zero on $D$. That means $h=\sum_{n}\alpha_{n}f_{n}$, where the $\alpha_{n}$ are i.i.d. standard Gaussian random variables and the $f_{n}$ are an orthonormal basis of the Hilbert space closure $H(D)$ of the space of $C^{\infty}$ real-valued bounded functions on $D$ with mean zero equipped with the Dirichlet inner product:
\[(f_{1},f_{2})_{\nabla}:=\frac{1}{2\pi}\int_{D}\nabla f_{1}(z)\cdot\nabla f_{2}(z)dz.\]

For $z\in\underline{\partial D}$, one can let $h_{\epsilon}(z)$ be the average value of $h$ on the semicircle of radius $\epsilon$ centered at $z$ and contained in $D$ (see \cite[Section 6.1]{SheffieldGFFMath} for a proof that makes sense of this). 

One may also consider the GFF $h$ on $D$ with \textit{mixed} boundary conditions. That is, $h$ has free boundary conditions on the linear component $\underline{\partial D}$, and zero boundary conditions on its complement $\partial D\backslash \underline{\partial D}$ (see the caption of Figure \ref{figure::symmetry} for the construction using a reflection principle). Let $h_{\epsilon}(z)$ be the semicircle average of $h$. For $z,z'\in\underline{\partial D}$ and small enough $\epsilon, \epsilon'>0$, \cite[Section 6.2]{DuplantierSheffieldLQGKPZ} showed that
\begin{equation}\label{equation::mix}
\mathrm{cov}(h_{\epsilon}(z), h_{\epsilon'}(z))=-2\log(\epsilon\vee\epsilon')-\tilde{G}_{z}(z),
\end{equation}
and if $B_{\epsilon}(z)\cap B_{\epsilon'}(z')=\emptyset$, then
\begin{equation}\label{equation::mixcov}\mathrm{cov}(h_{\epsilon}(z), h_{\epsilon'}(z'))=-2\log|z-z'|-\tilde{G}_{z}(z'),
\end{equation}
where $\tilde{G}_{z}$ is a harmonic function on $D$ satisfying certain boundary conditions. We remark that $\tilde G_{z}$ depends continuously on $z$ and (\ref{equation::mixcov}) implies that $\tilde{G}_{z}(z')=\tilde{G}_{z'}(z)$.
\begin{figure}[!htbp]
\begin{center}
\includegraphics{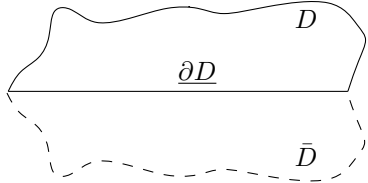}
\end{center}
\caption{\label{figure::symmetry}\small{In order to construct an instance of the GFF with mixed boundary conditions, we let $\bar{D}$ be the complex conjugate of $D$, and consider the whole domain $D^{\dagger}=D\cup\bar{D}$. The Hilbert space closure $H(D^{\dagger})$ of the space of $C^{\infty}$ real-valued functions compactly supported on $D^{\dag}$ can be decomposed as the direct sum $H_{e}(D^{\dagger})\oplus H_{o}(D^{\dagger})$ of the Hilbert space closures corresponding to even and odd functions on $D^{\dagger}$ with respect to the real line supporting $\underline{\partial D}$. The GFF $h$ on $D$ with mixed boundary conditions can be obtained by projecting the GFF on $D^{\dagger}$ with zero boundary conditions onto $H_{e}(D^{\dagger})$ and restricting to $D$. }}
\end{figure}
\begin{rk}\label{remark::freeboundary}
It is worthwhile to point out the relation between the free boundary GFF and the mixed boundary GFF. Recall the definition of $H(D)$ in the beginning of this section. We may also view $H(D)$ as the Hilbert space closure with respect to the Dirichlet inner product of the space of $C^{\infty}$ real-valued bounded functions on $D$ defined up to an additive constant. Let $H^{\dag}(D^{\dag})$ be the space of even functions (defined up to an additive constant) on $D^{\dag}$  with respect $\underline{\partial D}$ obtained by reflecting the functions in $H(D)$ with respect to $\underline{\partial D}$. It is easy to verify that $H^{\dag}(D^{\dag})$ has the following orthogonal decomposition with respect to the Dirichlet inner product 
\begin{equation}\label{equation::decom}
H^{\dag}(D^{\dag})=H_{0}^{\dag}(D^{\dag})\oplus H_{h}^{\dag}(D^{\dag}),
\end{equation}
where $H_{0}^{\dag}(D^{\dag})$ consists of the functions in $H^{\dag}(D^{\dag})$ which satisfy the zero boundary conditions on $\partial D^{\dag}$ (defined up to an additive constant), and $H_{h}^{\dag}(D^{\dag})$ consists of the functions in $H^{\dag}(D^{\dag})$ which are harmonic (defined up to an additive constant). By considering (\ref{equation::decom}) and restricting the functions to $D$, we observe that if $h$ is a free boundary GFF on $D$, it has the decomposition
\[h=h_{1}+h_{2},\]
where $h_{1}$ is a mixed boundary GFF on $D$ and $h_{2}$ is almost surely a harmonic function on $D$ (defined up to an additive constant) which can be extended to a harmonic function on $D^{\dag}$.
\end{rk}

For either boundary conditions, we let $\bar{h}_{\epsilon}(z):=\frac{\gamma^{2}}{4}\log\epsilon+\frac{\gamma}{2} h_{\epsilon}(z)$ and define the boundary measures $\mu_{\epsilon}^{B}:=e^{\bar{h}_{\epsilon}(z)}dz$, where $dz$ is the Lebesgue measure on the boundary component $\underline{\partial D}$. We are able to define the quantum boundary measures as the limit of $\mu^{B}_{\epsilon}$:

\begin{thm}\label{theorem::boundaryest}
Fix $\gamma\in[0,2)$ and define $D,\underline{\partial D},h,\mu_{\epsilon}^{B}$ as above with either free or mixed boundary conditions. Then as $\epsilon\rightarrow 0$, the measures $\mu_{\epsilon}^{B}$ a.s.\ converge to a limiting measure on $\underline{\partial D}$, which we denote by $\mu^{B}=\mu^{B}_{h}$. 
\end{thm}
\begin{proof}
We first consider the mixed boundary conditions. Fix $\gamma\in[0,2)$. Without loss of generality, we assume that $I=[0,1]\subsetneqq\underline{\partial D}$ and it suffices to show that the random variables $\mu_{\epsilon}^{B}(I)$ a.s. converge. For $y\in(0,1)$ and $k\ge 1$, let $S_{k}^{y}$ be the discrete set of $2^{k}$ points $a\in I$ with the property that $2^{k}a-2^{k}y\in\mathbb{Z}$. Define
\[A_{k}^{y}:=2^{-k}\sum_{z\in S_{k}^{y}}e^{\bar{h}_{2^{-k-1}}(z)}, \quad B_{k}^{y}:=2^{-k}\sum_{z\in S_{k}^{y}}e^{\bar{h}_{2^{-k-2}}(z)}.\]
Assume that $k$ is large enough such that all the semicircles considered are contained in $D$. We will give an estimate of $\mathbb{E}\left|A_{k}^{y}-B_{k}^{y}\right|^{2}$. With (\ref{equation::mix}) and (\ref{equation::mixcov}) in place of (\ref{eqaution::samecenter}) and (\ref{equation::variance}), 
a similar argument as in the proof of Lemma \ref{lemma::Nconvergence} implies that given the values of $h_{2^{-k-1}}(z)$ for $z\in S_{k}^{y}$, the conditional expectation of $\left|A_{k}^{y}-B_{k}^{y}\right|^{2}$ is
\[\mathbb{E}\left.\left(\left|A_{k}^{y}-B_{k}^{y}\right|^{2}\right|h_{2^{-k-1}}(z), z\in S_{k}^{y}\right)=2^{-2k}\tilde{C}\sum_{z\in S_{k}^{y}}e^{2\bar{h}_{2^{-k-1}}(z)},\]
where 
\[\tilde{C}=\mathbb{E}\left(\left.\left|1-2^{-\frac{\gamma^{2}}{4}}\exp\left[\frac{\gamma}{2} h_{2^{-k-2}}(z)\right]\right|^{2}\right|h_{2^{-k-1}}(z)=0\right)=e^{\frac{\gamma^{2}}{2}}-1\]
is a constant independent of $k$ and $z$.

Therefore, the unconditional expectation is 
\[\mathbb{E}\left|A_{k}^{y}-B_{k}^{y}\right|^{2}=\tilde{C}2^{-2k}2^{(k+1)\frac{\gamma^{2}}{2}}\sum_{z\in S_{k}^{y}}e^{-\frac{\gamma^{2}}{2}\tilde{G}_{z}(z)}\asymp 2^{-k(1-\frac{\gamma^{2}}{2})},\]
where the constant in $\asymp$ is independent of $k$ and $y$.
By Jensen's inequality and the Borel-Cantelli lemma, we have that as $\epsilon\rightarrow 0$ along negative powers of $2$, the random variables $\mu_{\epsilon}(I)$ a.s.\ converge.

The case for $\sqrt{2}\le\gamma<2$ can be proved using an argument similar to the proof of Lemma \ref{lemma::Nconvergence} by breaking the sum over $z\in S_{k}^{y}$ into two parts. Therefore, we are able to define the quantum boundary measure $\mu^{B}$ as the almost sure limit of $\mu^{B}_{2^{-k}}$ as $k\rightarrow\infty$.

Using the above setup and the same argument as that in the proof of Lemma \ref{lemma::Nconvergence}, it is easy to verify that for each integer $N\ge 1$, the random measures $\mu^{B}_{2^{-k/N}}$ a.s.\ converge to $\mu^{B}$ as $k\rightarrow\infty$.

In \cite[Section 6.2]{SheffieldGFFMath}, it was shown that given a reference radius $\epsilon_{0}$, for $0<\epsilon<\epsilon_{0}$, the Gaussian random variables $h_{\epsilon}(z)-h_{\epsilon_{0}}(z)$ is a standard Brownian motion $B_{t}$ independent of $h_{\epsilon_{0}}(z)$, where $t=-2\log(\epsilon/\epsilon_{0})$ . For each positive integer $N$, let 
\[\overline{C}(N):=\left(\mathbb{E}\sup_{t\in[0,\frac{2\log 2}{N}]}e^{-\frac{\gamma^{2}t}{8}}e^{\frac{\gamma}{2} B_{t}}\right)^{-1},\]
and
\[\underline{C}(N):=\left(\mathbb{E}\inf_{t\in[0,\frac{2\log 2}{N}]}e^{-\frac{\gamma^{2}t}{8}}e^{\frac{\gamma}{2} B_{t}}\right)^{-1}.\]
By considering the random measures 
\[\overline{C}(N)\sup_{\epsilon\in[2^{-(k+1)/N}, 2^{-k/N}]} e^{\bar{h}_{\epsilon}(z)}dz\quad\textrm{and}\quad\underline{C}(N)\inf_{\epsilon\in[2^{-(k+1)/N}, 2^{-k/N}]} e^{\bar{h}_{\epsilon}(z)}dz\]
as in the proof of Theorem \ref{theorem::measureconvergence}, we are able to show that as $\epsilon\rightarrow 0$, the random measures $\mu^{B}_{\epsilon}$ converge to $\mu^{B}$, which completes the proof for the mixed boundary GFF.

The result for the free boundary GFF follows from Remark \ref{remark::freeboundary} and the result of the mixed boundary GFF.
\end{proof}
An analogous transformation rule  to (\ref{equation::transformrule}) also holds for the quantum boundary measures:
\begin{thm}
Fix $\gamma\in[0,2)$ and define $D,\underline{\partial D},h,\mu^{B}$ as above with either free or mixed boundary conditions. Let $\Lambda$ be the collection of the conformal maps from bounded domains onto $D$ such that $\phi^{-1}(\underline{\partial D})$ is a line segment of the boundary on the real axis and $(\phi^{-1})'|_{\underline{\partial D}}>0$ (see Figure \ref{figure::map}). Then it is almost surely the case that for all $\phi\in\Lambda$, the measures $\mu^{B}_{\tilde{h}_{\phi}}$ with $\tilde{h}_{\phi}$ as defined in (\ref{equation::tildeh}) are well-defined. The transformation rule---
$\mu^{B}_{h}$ on $\underline{\partial D}$ is the image under $\phi$ of $\mu^{B}_{\tilde{h}_{\phi}}$ on $\phi^{-1}(\underline{\partial D})$---
almost surely holds simultaneously for all $\phi\in\Lambda$.
\end{thm}
\begin{figure}[!htbp]
\begin{center}
\includegraphics{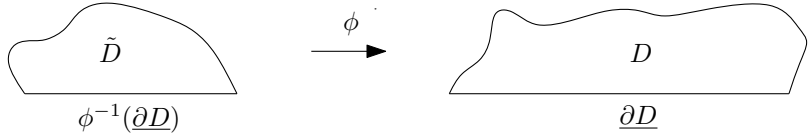}
\end{center}
\caption{\label{figure::map}\small{$\phi$ is a conformal map from $\tilde{D}$ to $D$. We assume that $\phi^{-1}(\underline{\partial D})\subset\partial\tilde{D}\cap\mathbb{R}$ is a line segment and $(\phi^{-1})'|_{\underline{\partial D}}>0$.}}
\end{figure}
\begin{proof}
We first consider the mixed boundary conditions. Recall the construction of $h$ (see the caption of Figure \ref{figure::symmetry}). By a reflection principle, we are able to extend $h$ on $D$ to $h^{\dag}$ on $D^{\dag}=D\cup\bar{D}$. \cite[Proposition 2.7]{SheffieldGFFMath} implies that $h^{\dag}$ is almost surely a continuous linear functional on the space $H_{s}(D^{\dag})\cap H_{e}(D^{\dag})$ of smooth functions which are compactly supported on $D^{\dag}$ and even with respect to the real axis. For $\rho_{1},\rho_{2}\in H_{s}(D^{\dag})\cap H_{e}(D^{\dag})$, we have
\[\mathrm{cov}((h^{\dag},\rho_{1}), (h^{\dag},\rho_{2}))=2\iint_{D^{\dag}\times D^{\dag}}G(x,y)\rho_{1}(x)\rho_{2}(y)dxdy,\]
where $G(x,y)$ is the Green's function for the Laplacian on $D^{\dag}$ as defined in (\ref{equation::green}). The conformal invariance of the Green's function implies that $h^{\dag}$ is also conformal invariant. We call $h^{\dag}$ the reflected extension of the mixed boundary GFF $h$ on $D$. We also observe that $h^{\dag}$ has the following Markov property. Suppose that $z\in\underline{\partial D}, r>0$ and $B_{r}(z)\subset D^{\dag}$. Given the values of $h^{\dag}$ outside of $B_{r}(z)$, the conditional law of $h^{\dag}$ on $B_{r}(z)$ is given by the sum of the harmonic extension of the values of $h^{\dag}$ on $\partial B_{r}(z)$ and  the reflected extension of an independent mixed boundary GFF on the semi-disk centered at $z$ with radius $r$ contained in $D$.
 
To prove the transformation rule, we will make use of the convolution techniques in Section \ref{section::bump}. For each mollifier $f_{\epsilon}$ given in (\ref{equation::bump}) and $z\in\underline{\partial D}$, we have that
\[h^{\dag}*f_{\epsilon}(z)=(h^{\dag}, f_{\epsilon}(z-\cdot))=2\pi\int_{0}^{\epsilon}h_{r}(z)f_{\epsilon}(r)dr.\]
Let
\[\tilde{h}_{\epsilon}(z):=\frac{\gamma}{2}h^{\dag}*f_{\epsilon}(z)+\frac{\gamma^{2}}{4}(\log\epsilon+C),\]
where $C$ is as defined in (\ref{equation::defC}). Then a similar argument to that in the proof of Theorem \ref{prop::convergence} implies that the random measures $\tilde{\mu}_{\epsilon}^{B}:=e^{\tilde{h}_{\epsilon}(z)}dz$ a.s.\ converge to $\mu^{B}$ as $\epsilon\rightarrow 0$. 

For each $\phi\in\Lambda$, we are able to extend $\phi^{-1}$ to a conformal map on $D^{\dag}$ by setting $\phi^{-1}(\bar{x})=\overline{\phi^{-1}(x)}$ for $x\in D$. It is easy to verify that an analogous result to Corollary \ref{lemma::differentradius} holds: the random measures $e^{\tilde{h}_{\epsilon/(\phi^{-1})'(z)}}dz$ converge to $\mu^{B}$ as $\epsilon\rightarrow 0$. Combining this fact with change of coordinates as in the proof of Theorem \ref{thm::conformalvariance}, we have that in order to obtain the desired result, it suffices to show that as $\epsilon\rightarrow 0$ along negative powers of 2, it is the case that
\begin{equation}\label{equation::bouest1}
(e^{C}\epsilon)^{\frac{\gamma^{2}}{4}}\left|\sup_{\phi\in\Lambda}\int_{I}\exp\left[\frac{\gamma}{2}(h^{\dag}\circ\phi)*f_{\epsilon(\phi^{-1})'(z)}(\phi^{-1}(z))\right]-\exp\left[\frac{\gamma}{2} h^{\dag}*f_{\epsilon}(z)\right]dz\right|\rightarrow 0,
\end{equation}
and
\begin{equation}\label{equation::bouest2}
(e^{C}\epsilon)^{\frac{\gamma^{2}}{4}}\left|\inf_{\phi\in\Lambda}\int_{I}\exp\left[\frac{\gamma}{2}(h^{\dag}\circ\phi)*f_{\epsilon(\phi^{-1})'(z)}(\phi^{-1}(z))\right]-\exp\left[\frac{\gamma}{2} h^{\dag}*f_{\epsilon}(z)\right]dz\right|\rightarrow 0
\end{equation}
almost surely, where we assume that $I=[0,1]\subsetneqq\underline{\partial D}$.

Let $\Lambda^{*}:=\left\{\phi\in\Lambda: \phi^{-1}(0)=0, (\phi^{-1})'(0)=1\right\}$ and $h^{*}$ be the reflected extension of a mixed boundary GFF on the semi-disk centered at $0$ with radius $4$. For $\phi\in\Lambda^{*}$, define $f^{\phi}_{\epsilon}$ as in (\ref{equation::fdef}). It is clear that $f_{\epsilon}, f^{\phi}_{\epsilon}\in H_{s}(B_{4}(0))\cap H_{e}(B_{4}(0))$. Since $h^{*}$ is almost surely a continuous linear functional on $H_{s}(B_{4}(0))\cap H_{e}(B_{4}(0))$, Lemma \ref{lemma::unicon} and the argument in the proof of Lemma \ref{lemma::est} imply that for any $\alpha>0$, we have
\begin{equation}\label{equation::mixest}
0\le\mathbb{E}e^{\alpha\left(\sup_{\phi\in\Lambda^{*}}(h^{*}, f_{\epsilon}^{\phi})-(h^{*},f)\right)}-1\lesssim g(\epsilon),
\end{equation}
where $g$ is a function such that $g(\epsilon)\rightarrow 0$ as $\epsilon\rightarrow 0$.

Applying the setup in the beginning of the proof of Theorem \ref{theorem::boundaryest} and the Markov property of $h^{\dag}$, we see that the proofs of (\ref{equation::bouest1}) and (\ref{equation::bouest2}) proceed exactly the same as the proofs of (\ref{equation::sup}) and (\ref{equation::inf}) with (\ref{equation::mixest}) in place of (\ref{equation::est}). This completes the proof of the mixed boundary GFF.

The result for the free boundary GFF follows from Remark \ref{remark::freeboundary} and the case of the mixed boundary GFF.
\end{proof}

\bibliographystyle{alpha}
\bibliography{commutativity_LQG}

\begin{thebibliography}{FLDR10}

\bibitem[Ber15]{berestycki2015elementary}
Nathana{\"e}l Berestycki.
\newblock An elementary approach to {G}aussian multiplicative chaos.
\newblock {\em arXiv preprint arXiv:1506.09113}, 2015.

\bibitem[Bor75]{MR0399402}
Christer Borell.
\newblock The {B}runn-{M}inkowski inequality in {G}auss space.
\newblock {\em Invent. Math.}, 30(2):207--216, 1975.

\bibitem[BSS14]{berestycki2014equivalence}
Nathana{\"e}l Berestycki, Scott Sheffield, and Xin Sun.
\newblock Equivalence of {L}iouville measure and {G}aussian free field.
\newblock {\em arXiv preprint arXiv:1410.5407}, 2014.

\bibitem[CIS76]{cirel1976norms}
BS~Cirel'soN, IA~Ibragimov, and VN~Sudakov.
\newblock Norms of gaussian sample functions.
\newblock In {\em Proceedings of the Third Japan—USSR Symposium on
  Probability Theory}, pages 20--41. Springer, 1976.

\bibitem[DS11a]{DuplantierSheffieldLQGKPZ}
Bertrand Duplantier and Scott Sheffield.
\newblock Liouville {Q}uantum {G}ravity and {KPZ}.
\newblock {\em Invent. Math.}, 185(2):333--393, 2011.

\bibitem[DS11b]{duplantier2011schramm}
Bertrand Duplantier and Scott Sheffield.
\newblock Schramm-{L}oewner {E}volution and {L}iouville {Q}uantum {G}ravity.
\newblock {\em Physical review letters}, 107(13):131305, 2011.

\bibitem[FLDR10]{fyodorov2010freezing}
Yan~V Fyodorov, Pierre Le~Doussal, and Alberto Rosso.
\newblock Freezing {T}ransition in {D}ecaying {B}urgers {T}urbulence and
  {R}andom {M}atrix {D}ualities.
\newblock {\em EPL (Europhysics Letters)}, 90(6):60004, 2010.

\bibitem[HK71]{MR0292433}
Raphael H{\o}egh-Krohn.
\newblock A general class of quantum fields without cut-offs in two space-time
  dimensions.
\newblock {\em Comm. Math. Phys.}, 21:244--255, 1971.

\bibitem[Hou95]{MR1354013}
Christian Houdr\'e.
\newblock Some applications of covariance identities and inequalities to
  functions of multivariate normal variables.
\newblock {\em J. Amer. Statist. Assoc.}, 90(431):965--968, 1995.

\bibitem[JS15]{junnila2015uniqueness}
Janne Junnila and Eero Saksman.
\newblock The uniqueness of the {G}aussian multiplicative chaos revisited.
\newblock {\em arXiv preprint arXiv:1506.05099}, 2015.

\bibitem[Kah60]{MR0117506}
J.-P. Kahane.
\newblock Propri\'et\'es locales des fonctions \`a s\'eries de {F}ourier
  al\'eatoires.
\newblock {\em Studia Math.}, 19:1--25, 1960.

\bibitem[Kah85]{MR829798}
Jean-Pierre Kahane.
\newblock Sur le chaos multiplicatif.
\newblock {\em Ann. Sci. Math. Qu\'ebec}, 9(2):105--150, 1985.

\bibitem[Man72]{mandelbrot1972possible}
Benoit~B Mandelbrot.
\newblock Possible refinement of the lognormal hypothesis concerning the
  distribution of energy dissipation in intermittent turbulence.
\newblock In {\em Statistical models and turbulence}, pages 333--351. Springer,
  1972.

\bibitem[Man74]{FLM:385758}
Benoit~B. Mandelbrot.
\newblock Intermittent turbulence in self-similar cascades: divergence of high
  moments and dimension of the carrier.
\newblock {\em Journal of Fluid Mechanics}, 62:331--358, 1 1974.

\bibitem[MS13]{miller2013quantum}
Jason Miller and Scott Sheffield.
\newblock Quantum {L}oewner {E}volution.
\newblock {\em arXiv preprint arXiv:1312.5745}, 2013.

\bibitem[RV10]{MR2642887}
Raoul Robert and Vincent Vargas.
\newblock Gaussian {M}ultiplicative {C}haos revisited.
\newblock {\em Ann. Probab.}, 38(2):605--631, 2010.

\bibitem[RV14]{MR3274356}
R{\'e}mi Rhodes and Vincent Vargas.
\newblock Gaussian multiplicative chaos and applications: a review.
\newblock {\em Probab. Surv.}, 11:315--392, 2014.

\bibitem[Sha16]{MR3475456}
Alexander Shamov.
\newblock On {G}aussian multiplicative chaos.
\newblock {\em J. Funct. Anal.}, 270(9):3224--3261, 2016.

\bibitem[She07]{SheffieldGFFMath}
Scott Sheffield.
\newblock Gaussian free fields for mathematicians.
\newblock {\em Probab. Theory Related Fields}, 139(3-4):521--541, 2007.

\bibitem[She10]{sheffield2010conformal}
Scott Sheffield.
\newblock Conformal weldings of random surfaces: {SLE} and the quantum gravity
  zipper.
\newblock {\em arXiv preprint arXiv:1012.4797}, 2010.

\end{thebibliography}
\end{document}